\documentclass[compress, preprint,9pt]{elsarticle}

\usepackage[abs]{overpic}
\usepackage{graphicx}
\usepackage{subfigure}
\usepackage{stmaryrd}
\usepackage{amsfonts, color, amssymb}
\usepackage{amsmath}
\usepackage{amsthm}
\usepackage{epsfig}
\usepackage{psfrag}
\usepackage{bm}
\usepackage{paralist}
\usepackage{algorithm}
\usepackage{algorithmicx}
\usepackage{algpseudocode}
\usepackage{appendix}
\usepackage{caption}
\usepackage{hyperref}
\usepackage{color}

\usepackage[misc]{ifsym}

\usepackage[a4paper, total={6in,8in}]{geometry}

\allowdisplaybreaks[4]
\newtheorem{theorem}{Theorem}[section] %
\newtheorem{example}{Example}[section]

\newtheorem{lemma}{Lemma}[section]
\def\3bar{{|\hspace{-.02in}|\hspace{-.02in}|}}

\allowdisplaybreaks[4]

\allowdisplaybreaks 

\numberwithin{equation}{section}

\journal{Journal of Computational and Applied Mathematics}

\begin{document}

\begin{frontmatter}

\title{A posteriori error estimation for weak Galerkin method of the fourth-order singularly perturbed problem}

\author[mymainaddress]{Shicheng Liu}
\ead{lsc22@mails.jlu.edu.cn}

\author[mymainaddress]{Qilong Zhai\corref{mycorrespondingauthor}}
\cortext[mycorrespondingauthor]{Corresponding author}
\ead{zhaiql@jlu.edu.cn}

\address[mymainaddress]{School of Mathematics, Jilin University, Changchun {130012}, China}

\begin{abstract}
	In this paper, we present a posteriori error estimation for weak Galerkin method applied to fourth order singularly perturbed problem. The weak Galerkin discretization space and numerical scheme are first described. A fully computable residual type error estimator is then constructed. Both the reliability and efficiency of the proposed estimator are rigorously demonstrated. Numerical experiments are provided to validate the theoretical findings.
\end{abstract}

%


\end{frontmatter}


\section{Introduction}
\label{section:introduction}
For given a bounded domain $\Omega \subset \mathbb{R}^{2}$ and $f\in L^{2}(\Omega)$, we consider the following fourth-order singularly perturbed elliptic boundary value problem
\begin{align}
  \varepsilon ^{2}\Delta ^{2}u-\Delta u&=f, \quad \text{in}~\Omega,\label{1.1}\\
  u= \nabla u \cdot \mathbf{n} &=0, \quad \text{on}~\partial\Omega.\label{1.2}
\end{align}
In singularly perturbed models, the parameter $0 < \varepsilon \ll 1$ is a non-negative real number conventionally referred to as the singular perturbation parameter. The boundary value problem (\ref{1.1})-(\ref{1.2}) arises in the linear elasticity modeling of sufficiently thin buckling plates, where $u$ represents the displacement in a clamped plate model. The parameter $\varepsilon$, assumed to be small enough, is defined by $\varepsilon = t^{3} E/12(1- \nu^{2})\iota^{2}S$, where $t$ denotes the plate thickness, $E$ is Young's modulus of the elastic material, $\nu$ is the Poisson ratio, $iota$ represents the characteristic diameter of the plate, and $S$ denotes the measure of the density of the isotropic stretching force.

The numerical analysis of fourth-order singularly perturbed problems has been the subject of extensive research within the scientific community \cite{06SPP,21SPP,27SPP,32SPP,34SPP}. Meng and Stynes \cite{26SPP} investigated the Adini finite element method for such problems on a Shishkin mesh in the context of fourth-order problems. The $C^{0}$ interior penalty finite element method has been developed for fourth-order singularly perturbed problems in \cite{03SPP,18SPP}. Constantinou et al. proposed an hp-finite element method to solve these problems in \cite{12SPP,13SPP}, while the convergence of a mixed finite element method was examined in \cite{16SPP}. Guo et al. cite{19SPP} analyzed a standard $C^{1}$-conforming finite element method of polynomial degree $p$ on a one-dimensional mesh. Furthermore, Franz et al. \cite{17SPP} established error estimates in a balanced norm for finite element methods applied to higher-order reaction-diffusion problems.

Over recent decades, computation with adaptive grid refinement has established itself as a valuable and efficient methodology in scientific computing. Central to this technique is the design of an accurate a posteriori error estimator, which offers guidance on where and how to refine the grid. An estimator is deemed reliable if it provides a rigorous upper bound for the exact error, and efficient if it furnishes a corresponding lower bound. Upper bounds combined with lower bounds yield error indicators of optimal order, enabling efficient mesh refinement. Computable a posteriori error estimates and adaptive strategies for fourth-order problems have garnered growing interest over the last twenty years. For examples, the conforming approximations of problems involving the biharmonic operator of \cite{46Postertiori}, the treatment of Morley plates \cite{09Postertiori,33Postertiori}, quadratic $C^0$-conforming interior penalty methods \cite{14Postertiori} and general order discontinuous Galerkin methods \cite{28Postertiori} for the biharmonic problem, continuous and discontinuous Galerkin approximations of the Kirchhoff-Love plate \cite{30Postertiori}, the dichotomy principle in a posteriori error estimates for fourth-order problems \cite{02Postertiori}, and the Ciarlet-Raviart formulation of the first biharmonic problem \cite{20Postertiori}.

The weak Galerkin (WG) method has proven to be an effective numerical technique for solving partial differential equations. It was initially introduced by Junping Wang and Xiu Ye in \cite{wysec2} for second order elliptic problems. The fundamental idea of the WG method lies in constructing separate approximation functions on the interior and the boundary of each mesh cell, and replacing the classical differential operator with a discretized weak differential operator. The WG method has been successfully applied to Stokes equations \cite{wy1302,WangZhaiZhangS2016,wangwangliu2022}, elasticity equations \cite{chenxie2016,Liu,lockingw,Yisongyang}, Maxwell's equations \cite{MLW14}, biharmonic equations \cite{MuWangYeZhang14,ZhangZhai15}, Navier-Stokes equations \cite{HMY2018,LLC18,zzlw2018}, Brinkman equations \cite{MuWangYe14,WangZhaiZhang2016,ZhaiZhangMu16}, as well as in the contexts of the multigrid approach \cite{CWWY15} and the maximum principle \cite{maximumwang2,WangYeZhaiZhang18}. In addition, the WG method has produced promising results for singularly perturbed problems, such as one and two dimensional convection-diffusion problems \cite{WG_1D_Bakhvalov,WG_1D_Shishkin,WG_2D_Shishkin}and the singularly perturbed biharmonic equation on uniform meshes \cite{singularly_perturbed_biharmonic}.

The goal of this paper is to develop a residual based a posteriori error estimator within the WG method for fourth-order singularly perturbed problems, and to establish its theoretical reliability and efficiency. This paper is organized as follows. In Section 2, we introduce the Shishkin mesh and the assumptions associated. In Section 3, we give the definitions of the weak Laplacian operator and weak gradient operator. We also present WG finite element schemes for the singularly perturbed value problem. In Section 4, we introduce some local $L^2$ projection operators and give some approximation properties. In Section 5, we establish error estimates for the WG scheme in a $H^2$-equivalent discrete norm. And in Section 6, we report the results of two numerical experiments.


\section{Preliminaries and notations}
Let $\mathcal{T}_{h}$ be a non-overlapping polygonal mesh of the domain $\Omega$. For each cell $T \in \mathcal{T}_{h}$, let $h_T$ denote its diameter and define the global mesh size as $\displaystyle h = \max_{T \in \mathcal{T}_h}h_T$. The area of $T$ is denoted by $|T|$, and $\partial T$ represents the set of all edges of $T$. Let $\mathcal{E}_h^0$ denote the set of all interior edges and $\mathcal{E}_h^b$ the set of all boundary edges satisfying $\mathcal{E}_h^b \subset \partial\Omega$. 

For an interior edge $e \in \mathcal{E}_{h}^{0}$ shared by two adjacent cells $T_{+}$ and $T_{-}$, let $\mathbf{n}_{e}$ be the unit normal vector pointing from $T_{+}$ to $T_{-}$. The average and jump of a function $v$ across $e$ are defined as
\begin{align*}
  &\{v\} = \frac{1}{2}(v_{+} + v_{-}), &&[v] = v_{+} - v_{-},
\end{align*}
where $v_{+}$ and $v_{-}$ denote the traces of $v$ on $e$ from $T_{+}$ and $T_{-}$, respectively. For a boundary edge $e \in \mathcal{E}_h^b$, these operators are defined as $$\{v\} = [v] = v.$$

For any integer $k \geq 2$, the local discrete weak function space on a cell $T$ is defined as 
\begin{equation*}
    W_{k}(T)=\left\{v_h=\{v_0, v_b, \mathbf{v}_g\}: v_0\in \mathbb{P}_{k}(T), v_b\in \mathbb{P}_{k}(e), \mathbf{v}_g\in [\mathbb{P}_{k}(e)]^{2}, e \subset \partial T \right\}.
\end{equation*}
The global weak finite element space is then defined by
$$V_h=\left\{v_h: v_h|_T\in W_k(T), \quad\forall T\in\mathcal{T}_{h}\right\}.$$
Let $V_{h}^{0}$ denote the subspace of $V_h$ consisting of functions with vanishing traces on the boundary 
$$V_{h}^{0}=\left\{v_h\in V_h, v_b|_e=0, \mathbf{v}_g\cdot\mathbf{n}|_e=0, e\subset\partial T\cap \partial\Omega\right\}.$$
Furthermore, define the space of interior functions as
$$\mathbb{V}_{h}^{I}=\left\{v_{0}: v_{0} \text{~is the interior component of~} v_{h} \in V_{h} \right\}.$$

For any $v \in V_{h} + H^{2}(\Omega)$, the discrete weak Hessian operator $D^2_{w, k}$ is defined on each cell $T$ as the unique polynomial in $[\mathbb{P}_{k}(T)]^{2\times 2}$ satisfying
\begin{align}\label{3.1}
  \left(D^2_{w,k}v, \widetilde{\varphi} \right)_T=\left(v_0, \nabla \cdot \nabla \cdot \widetilde{\varphi}\right)_T-\left\langle v_b, \nabla \cdot \widetilde{\varphi} \cdot\mathbf{n}\right\rangle+\left\langle \mathbf{v}_g, \widetilde{\varphi} \mathbf{n}\right\rangle, \quad\forall \widetilde{\varphi} \in [\mathbb{P}_k(T)]^{2\times 2},
\end{align}
where $\mathbf{n}$ denotes the unit outward normal vector to \( \partial T \). Similarly, the discrete weak gradient $\nabla_{w,k} v$ is defined as the unique polynomial in $[\mathbb{P}_k(T)]^2$ such that
\begin{align}\label{3.2}
  \left(\nabla_{w,k}v, \mathbf{q}\right)_T=-\left(v_0, \nabla\cdot \mathbf{q}\right)_T+\left\langle v_b, \mathbf{q}\cdot\mathbf{n}\right\rangle_{\partial T}, \quad\forall \mathbf{q}\in \left[\mathbb{Q}_k(T)\right]^{2}. 
\end{align} 
For notational simplicity, and when no confusion may arise, we omit the subscript $k$ in $D^2_{w,k}$ and $\nabla_{w,k}$, denoting them simply as $D^2_w$ and $\nabla_w$, respectively. 

For each cell $T$, let $\mathcal{Q}_0$ denote the $L^2$-projection onto $\mathbb{P}_k(T)$. On each edge $e \subset \partial T$, define the edge-based $L^2$-projections $\mathcal{Q}_b$ onto $\mathbb{P}_k(e)$ and $\mathbf{Q}_g$ onto $[\mathbb{P}_k(e)]^2$. In addition, let $\mathbf{Q}_h$ and $\mathbb{Q}_h$ be the local $L^2$-projection operators onto $[\mathbb{P}_k(T)]^2$ and $[\mathbb{P}_k(T)]^{2 \times 2}$, respectively. We then define a projection $\mathcal{Q}_h u$ into the finite element space $V_h$ component-wise on each element $T$ as
\begin{align*}
  \mathcal{Q}_h u=\left\{\mathcal{Q}_0 u, \mathcal{Q}_b u, \mathbf{Q}_g(\nabla u)\right\}.
\end{align*}
\begin{lemma} \label{lemma2.1}
  On each cell $T \in \mathcal{T}_{h}$, the following identity holds for all $v \in H^2(T)$,
  \begin{align}\label{4.1}
    D^2_{w} v=\mathbb{Q}_h(D^2 v).
  \end{align}
\end{lemma}
\begin{proof}
  For any $\widetilde{\varphi} \in [\mathbb{P}_{k}(T)]^{2\times 2}$, it follows that
  \begin{align*}
    \left(D^2_{w} v,\widetilde{\varphi}\right)_{T}&=\left(v, \nabla\cdot\nabla\cdot \widetilde{\varphi} \right)_{T}-\left\langle v,\nabla\cdot \widetilde{\varphi} \cdot\mathbf{n}\right\rangle_{\partial T}+\left\langle \nabla v, \widetilde{\varphi} \mathbf{n} \right\rangle_{\partial T}\\
    &=\left(D^2 v, \widetilde{\varphi}\right)_{T}\\
    &=\left(\mathbb{Q}_{h}(D^2 v), \widetilde{\varphi}\right)_{T}.
  \end{align*}
  Since the equality holds for all $\widetilde{\varphi} \in [\mathbb{P}_{k}(T)]^{2\times 2}$, we conclude that
  \begin{align*}
    D^2_{w} v=\mathbb{Q}_h(D^2 v),
  \end{align*}
  which completes the proof.
\end{proof}
\begin{lemma} \label{lemma2.2}
  On each cell $T \in \mathcal{T}_{h}$, the following identity holds for all $v \in H^2(T)$, 
  \begin{align}\label{4.2}
    \nabla_{w} v=\mathbf{Q}_{h}(\nabla v).
  \end{align}
\end{lemma}
\begin{proof}
  For any $\mathbf{q} \in [\mathbb{P}_{k}(T)]^2$, applying integration by parts yields
  \begin{align*}
    \left(\nabla_{w} v, \mathbf{q}\right)_{T}&=-\left(v, \nabla\cdot \mathbf{q}\right)_{T}+\left\langle v, \mathbf{q}\cdot\mathbf{n}\right\rangle_{\partial T}\\
    &=\left(\nabla v, \mathbf{q}\right)_{T}\\
    &=\left(\mathbf{Q}_{h}\nabla v, \mathbf{q}\right)_{T}.
  \end{align*}
  Since the equality holds for all $\mathbf{q} \in [\mathbb{P}_{k}(T)]^2$, we conclude that
  \begin{align*}
    \nabla_{w} v=\mathbf{Q}_{h}(\nabla v),
  \end{align*}
  which completes the proof.
\end{proof}

\section{Numerical algorithm}
For notational simplicity, we employ the following conventions
\begin{align*}
  &\left(D^2_{w}u_h, D^2_{w}v_h\right)_{\mathcal{T}_{h}}=\sum_{T\in\mathcal{T}_h}\left(D^2_{w}u_h,  D^2_{w}v_h\right)_T,
  &&\left\Vert D^2_{w}u_h \right\Vert_{\mathcal{T}_{h}}^{2}=\sum_{T\in\mathcal{T}_h}\left\Vert D^2_{w}u_h \right\Vert_T^{2},\\
  &\left(\nabla_{w}u_h, \nabla_{w}v_h\right)_{\mathcal{T}_{h}}=\sum_{T\in\mathcal{T}_h}\left(\nabla_{w}u_h, \nabla_{w}v_h\right)_T,
  &&\left\Vert \nabla_{w}u_h \right\Vert_{\mathcal{T}_{h}}^{2}=\sum_{T\in\mathcal{T}_h}\left\Vert \nabla_{w}u_h \right\Vert_T^{2}.
\end{align*}
The weak Galerkin scheme is then formulated as follows. Find $u_h \in V_{h}^{0}$ such that 
\begin{align}\label{3.3}
  B_{h}(u_{h}, v_{h}) + A_{h}(u_{h}, v_{h}) = (f, v_{0}), &&\forall v_{h} \in V_h^0,
\end{align}
where 
\begin{align*}
    &B_{h}(u_{h}, v_{h}) = \varepsilon^{2}(D^2_{w}u_{h}, D^2_{w}v_h)_{\mathcal{T}_{h}}+S_{1}(u_{h},v_{h}),\\
    &A_{h}(u_{h}, v_{h}) = (\nabla_{w}u_{h},\nabla_{w}v_{h})_{\mathcal{T}_{h}}+S_{2}(u_{h},v_{h}),
\end{align*}
and the stabilizer terms are given by
\begin{align*}
  &S_{1}\left(u_h,v_{h}\right) = \sum_{T\in\mathcal{T}_h} \left(\varepsilon^{2} h_{T}^{-1}\left\langle\nabla u_0-\mathbf{u}_g,\nabla v_0-\mathbf{v}_g\right\rangle_{\partial T}+\varepsilon^{2}h_{T}^{-3} \left\langle u_0-u_b,v_0-v_b\right\rangle_{\partial T} \right),\\
  &S_{2}\left(u_h,v_{h}\right) = \sum_{T\in\mathcal{T}_h} \left(h_{T}\left\langle\nabla u_0-\mathbf{u}_g,\nabla v_0-\mathbf{v}_g\right\rangle_{\partial T}+h_{T}^{-1} \left\langle u_0-u_b,v_0-v_b\right\rangle_{\partial T} \right).
\end{align*}
We endow the WG space $V_{h}$ with an $H^2$-like seminorm defined by
\begin{align}\label{3.4}
    \3bar v_{h} \3bar^2 = B_{h}(v_{h}, v_{h}) + A_{h}(v_{h}, v_{h}).
\end{align}
\begin{lemma}
  The quantity $\3bar \cdot \3bar$ defines a norm in $V^0_h$, and the weak Galerkin scheme (\ref{3.3}) has a unique solution.
\end{lemma}
\begin{proof}
  To verify that $\3bar \cdot \3bar$ defines a norm on $V_h^0$, suppose $v_h \in V_h^0$ satisfies $\3bar v_{h} \3bar = 0$. It follows that $D^2_w v_h = 0$ and $\nabla_w v_h = 0$ on each cell $T$, together with the conditions $v_0 = v_b$ and $\nabla v_0 = \mathbf{v}_g$ on $\partial T$. Now, for any $\widetilde{\varphi} \in [\mathbb{P}_k(T)]^{2 \times 2}$, applying definition (\ref{3.1}) along with $D^2_w v_h = 0$, we obtain 
  \begin{equation}\label{3.5}
    \begin{split}
      0 &= \left(D^2_{w}v_{h},\widetilde{\varphi}\right)_T\\
        &= \left(v_0, \nabla\cdot\nabla\cdot\widetilde{\varphi}\right)_T-\left\langle v_b, \nabla\cdot\widetilde{\varphi} \cdot\mathbf{n}\right\rangle_{\partial T}+\left\langle\mathbf{v}_g, \widetilde{\varphi} \mathbf{n}\right\rangle_{\partial T}\\
        &= \left(D^2 v_0, \widetilde{\varphi}\right)_T + \left\langle v_0-v_b, \nabla\cdot\widetilde{\varphi} \cdot\mathbf{n}\right\rangle_{\partial T} - \left\langle\nabla v_0 -\mathbf{v}_g, \widetilde{\varphi} \mathbf{n}\right\rangle_{\partial T}\\
        &= \left(D^2 v_0, \widetilde{\varphi}\right)_T,
  \end{split}
  \end{equation}
  which implies $D^2 v_0 = 0$ on each cell $T$. Hence, $v_0$ is a linear polynomial on $T$, and $\nabla v_0$ is constant per cell. Combining this with the condition  $\nabla v_0 = \mathbf{v}_g$ on $\partial T$, it follows that $\nabla v_0$ is continuous across the entire domain $\Omega$. Since $\mathbf{v}_g = 0$ on $\partial \Omega$, we conclude that $\nabla v_0 = 0$ in $\Omega$ and $\mathbf{v}_g = 0$ on every edge. Therefore, $v_0$ is constant on each cell. Using the condition $v_0 = v_b$ on $\partial T$, we deduce that $v_0$ is continuous throughout $\Omega$. The boundary condition $v_b = 0$ on $\partial \Omega$ then implies $v_0 = 0$ in $\Omega$ and $v_b = 0$ on all edges.

  Now, let $u_h^{(1)}$ and $u_h^{(2)}$ be two distinct solutions of the WG scheme (\ref{3.3}). Then, the error $u_h^{(1)} - u_h^{(2)}$ belongs to $V_h^0$ and satisfies 
  \begin{align}\label{unique}
    B_{h}(u_h^{(1)} - u_h^{(2)}, v_{h}) + A_{h}(u_h^{(1)} - u_h^{(2)}, v_{h}) = 0, \quad \forall v \in V_h^0.
  \end{align}
  By choosing $v_h = u_h^{(1)} - u_h^{(2)}$ in equation (\ref{unique}), it follows that 
  \begin{align*}
    \3bar u_h^{(1)} = u_h^{(2)} \3bar^{2} = 0.
  \end{align*}
  Since $\3bar \cdot \3bar$ defines a norm on $V_h^0$, the identity $\3bar u_h^{(1)} = u_h^{(2)} \3bar = 0$ implies $u_h^{(1)} - u_h^{(2)} \equiv 0$, and hence $u_h^{(1)} = u_h^{(2)}$. This establishes the uniqueness of the solution.
\end{proof}

\section{A posteriori error estimation}
In this section, we introduce a residual-based a posteriori error estimator for the singularly perturbed problem and establish its reliability and efficiency. First, we define an energy norm $\Vert \cdot \Vert_{\varepsilon}$, balanced with respect to the singular perturbation parameter $\varepsilon$, on the space $H_{0}^{2}(\Omega)$ as follows 
\begin{align*}
  \Vert w \Vert_{\varepsilon} = \left(\varepsilon^2 \vert w \vert_{2}^{2} + \vert w \vert_{1}^{2}\right)^{1/2}.
\end{align*}
We also define the localized version $\Vert \cdot \Vert_{\varepsilon,d}$ of this norm over a subdomain $d$, which may be an edge $e$ or a cell $T$, as needed in the analysis.

To define an a posteriori error estimator for the singularly perturbed problem, we introduce the following local error indicators. First, define the parameters
\begin{align*}
  &\alpha_{T}=\min\{\varepsilon^{-1} h_{T}^{2}, h_{T}\}, 
  &&\alpha_{e,1}=\min\{\varepsilon^{-1} h_{T}^{3/2}, h_{T}^{1/2}\}, 
  &&\alpha_{e,2}=\min\{\varepsilon^{-1} h_{T}^{1/2}, h_{T}^{-1/2}\}.
\end{align*}
The local residual and jump terms are given by
\begin{align*}
  &R_{T} = f_{h} - \nabla\cdot\nabla\cdot \varepsilon^2 D^{2}_{w} u_{h} + \nabla\cdot \nabla_{w}u_{h},\\
  &J_{e,1} = [(\nabla\cdot \varepsilon^2 D_{w}^{2}u_{h} - \nabla_{w}u_{h}) \cdot \mathbf{n}],\\
  &J_{e,2} = [\varepsilon^2 D_{w}^{2}u_{h} \mathbf{n}].
\end{align*}
Then, the local error indicators are defined as
\begin{align*}
  &\eta_{T,1}^2= \alpha_{T}^{2} \left\Vert f - f_{h} \right\Vert_{T}^{2}, &&\eta_{T,2}^2= \alpha_{T}^{2} \left\Vert R_{T} \right\Vert_{T}^{2},\\
  &\eta_{e,1}^2=\sum_{e\in \partial T} \alpha_{e,1}^{2} \left\Vert J_{e,1} \right\Vert_{e}^{2}, &&\eta_{e,2}^2=\sum_{e\in \partial T} \alpha_{e,2}^{2} \left\Vert J_{e,2} \right\Vert_{e}^{2}.
\end{align*}
The local and global error estimators are respectively defined as
\begin{align*}
  &\eta_{T}= \left(\sum_{i=1}^{2} \left( \eta_{T,i}^{2}+\eta_{e,i}^{2}+S_{i}(u_{h},u_{h})|_{T} \right)\right)^{1/2}, &&\eta_{h}=\left(\sum_{T\in\mathcal{T}_{h}}\eta_{T}^{2}\right)^{1/2}.
\end{align*}

\subsection{Upper bound}
Our a posteriori error analysis employs a recovery operator $E$, introduced in \cite{28Postertiori}, which maps the space $\mathbb{V}_{h}^{I}$ into a $C^1$-conforming space $\mathbb{V}_{h}^{C} \subset C^1(\mathcal{T}_{h})$ constructed via macro elements of degree $k+2$. For a cell $T \in \mathcal{T}_h$, the macro element $\widetilde{\mathbb{P}}_{m}$ is defined over a subdivision of $T$ into subtriangles $\kappa_{1}$, $\kappa_{2}$, $\cdots$, $\kappa_{s}$, 
\begin{align*}
  \widetilde{\mathbb{P}}_{m} = \left\{ v \in C^1(T): v|_{\kappa_i} \in \mathbb{P}_{m}(\kappa_i), i = 1, 2, \cdots, s \right\}.
\end{align*}
Further details can be found in \cite{28Postertiori}. Adapted to our method, the recovery operator satisfies the following estimate.
\begin{lemma}\label{ROP}
There exists an operator $E: \mathbb{V}_{h}^{I} \to \mathbb{V}_{h}^{C} \cap H_0^2(\Omega)$ such taht  
\begin{align}
  \sum_{T \in \mathcal{T}_h} |u_0 - E(u_0)|_{\alpha, T}^2 \leq C \sum_{e \in \mathcal{E}_h}\left(\Vert h_{e}^{1/2-\alpha}[u_0]\Vert_{e}^2 + \Vert h_{e}^{3/2-\alpha}[\nabla u_0]\Vert_{e} \right),
\end{align}
for $\alpha = 0$, $1$, $2$, where $C>0$ is a constant independent of $h_{e}$ and $u_0$.
\end{lemma}
Let $u_{c} = E(u_{0})$. we decompose the error as follows
\begin{align}
  e_{h} = u - u_h = (u - u_{c}) + (u_{c} - u_h) = e_c + e_d.
\end{align}
\begin{lemma}\label{uplemma4.2} 
  Let $u\in H^2_0(\Omega)$ and $u_h \in V_{h}^{0}$ be the solutions to (\ref{1.1})-(\ref{1.2}) and (\ref{3.3}), respectively. Then we have
  \begin{align}
    l_{1}(u_{h}) \leq& C \eta_{h} \Vert e_c \Vert_{\varepsilon},
  \end{align}
  where
  \begin{align*}
    l_{1}(u_{h}) = \varepsilon^{2} \left( D^{2} u - D_{w}^{2}u_{h}, D^{2} u - D_{w}^{2}u_c \right)_{\mathcal{T}_{h}} + \left( \nabla u - \nabla_{w}u_{h}, \nabla u - \nabla_{w}u_c \right)_{\mathcal{T}_{h}}.
  \end{align*}
\end{lemma}
\begin{proof}
  For $v\in \mathbb{P}_{k}(\mathcal{T}_{h})\cap C(\overline{\Omega})$, the discrete weak Hessian is defined as
  \begin{align*}
    \left(D_{w}^{2}v, \widetilde{\varphi}\right)_{T} = \left(v, \nabla\cdot\nabla\cdot\widetilde{\varphi}\right)_{T} - \left\langle v, \nabla\cdot\widetilde{\varphi} \cdot \mathbf{n} \right\rangle_{\partial T} + \left\langle \{v\}, \widetilde{\varphi}\mathbf{n} \right\rangle_{\partial T}.
  \end{align*}
  Since $u$ is the solution to the weak formulation and $e_{c} \in H_{0}^{2}(\Omega)$, the following equation holds 
  \begin{align}\label{up4.3}
    \left( D^{2}u, D^{2}e_c \right)_{\mathcal{T}_{h}} + \left( \nabla u, \nabla e_c \right)_{\mathcal{T}_{h}} = \left( f, e_c \right)_{\mathcal{T}_{h}}.
  \end{align}
  Combining (\ref{up4.3}) with the WG scheme (\ref{3.3}), we expand $l_{1}(u_h)$ as follows
  \begin{align*}
    l_{1}(u_{h}) =& \sum_{T\in\mathcal{T}_{h}} \left( \varepsilon^{2} \left( D^{2}u, D^{2}e_c \right)_{T} - \varepsilon^{2} \left( D_{w}^{2}u_{h}, D_{w}^{2}e_c \right)_{T} + \left( \nabla u, \nabla e_c \right)_{T} - \left( \nabla_{w}u_{h}, \nabla_{w}e_c \right)_{T} \right)\\
    =& \sum_{T\in\mathcal{T}_{h}} \left( \left( f, e_c \right)_{T} - \varepsilon^{2} \left( D_{w}^{2}u_{h}, D_{w}^{2}e_c \right)_{T} - \left( \nabla_{w}u_{h}, \nabla_{w}e_c \right)_{T} \right)\\
    &+ S_{1}(u_{h}, Ie_{c}) + S_{2}(u_{h}, Ie_{c}) \left.\right)\\
    =& \sum_{T\in\mathcal{T}_{h}} \left(\right. \left( f-f_{h}, e_c - Ie_{c} \right)_{T} + \left( f_{h} - \nabla\cdot\nabla\cdot \varepsilon^{2}D_{w}^{2}u_{h} + \nabla\cdot \nabla_{w}u_{h}, e_c - Ie_{c} \right)_{T}\\
    &+ \left\langle \left(\nabla\cdot \varepsilon^{2} D_{w}^{2}u_{h} - \nabla_{w}u_{h}\right) \cdot \mathbf{n}, e_c-Ie_{c} \right\rangle_{\partial T} - \left\langle D_{w}^{2}u_{h} \mathbf{n}, \nabla e_c-\{\nabla Ie_{c}\} \right\rangle_{\partial T}\\
    &+ S_{1}(u_{h}, Ie_{c}) + S_{2}(u_{h}, Ie_{c}) \left.\right).
  \end{align*}
  where $I$ denotes the Lagrange linear interpolant and $Ie_{c} \in \mathbb{P}_{k}(\mathcal{T}_{h})\cap C(\overline{\Omega})$. The following approximation properties hold
  \begin{align*}
    &\left\Vert e_c - Ie_{c} \right\Vert_{T} \leq \varepsilon^{-1} h_{T}^{2} \varepsilon\left\vert e_c \right\vert_{2,T},
    &&\left\Vert e_c - Ie_{c} \right\Vert_{T} \leq h_{T} \left\vert e_c \right\vert_{1,T}.
  \end{align*}
  which together imply 
  \begin{align}\label{I_T}
    \left\Vert e_c - Ie_c \right\Vert_{T} \leq \alpha_{T} \left\Vert e_{c} \right\Vert_{\varepsilon,T}.
  \end{align}
  Furthermore, on each edge $e \subset \partial T$, applying the trace inequality and the inverse inequality yields
  \begin{align}\label{I_e0}
    \left\Vert e_c - Ie_c \right\Vert_{e} \leq \left\Vert e_c - Ie_c \right\Vert_{T_{e}}^{1/2} \left\vert e_c - Ie_c \right\vert_{1,T_{e}}^{1/2} \leq \alpha_{e,1} \left\Vert e_{c} \right\Vert_{\varepsilon,T_{e}}.
  \end{align}
  Similarly,
  \begin{align}\label{I_e1}
    \left\Vert \nabla (e_c- Ie_{c}) \right\Vert_{e} \leq \left\Vert \nabla (e_c- Ie_{c}) \right\Vert_{T_{e}}^{1/2} \left\vert \nabla (e_c- Ie_{c}) \right\vert_{1,T_{e}}^{1/2} \leq \alpha_{e,2} \left\Vert e_{c} \right\Vert_{\varepsilon,T_{e}}.
  \end{align}
  where $\alpha_{e,1}$ and $\alpha_{e,2}$ are as defined previously.

  For the first term of $l_{1}(u_{h})$, applying the Cauchy-Schwarz inequality and the interpolation estimate (\ref{I_T}) yields
  \begin{align*}
    \sum_{T\in\mathcal{T}_{h}} \left( f-f_{h}, e_c - Ie_{c} \right)_{T} \leq \sum_{T\in\mathcal{T}_{h}} \left\Vert f-f_{h} \right\Vert_{T} \left\Vert e_c - Ie_{c} \right\Vert_{T} \leq \sum_{T\in\mathcal{T}_{h}} \alpha_{T} \left\Vert f-f_{h} \right\Vert_{T} \left\Vert e_{c} \right\Vert_{\varepsilon,T}.
  \end{align*}
  For the second term, again using the Cauchy-Schwarz inequality and (\ref{I_T}), we obtain 
  \begin{align*}
    \sum_{T\in\mathcal{T}_{h}} \left( e_c - Ie_{c}, f_{h} - \nabla\cdot\nabla\cdot \varepsilon^{2}D_{w}^{2}u_{h} + \nabla\cdot \nabla_{w}u_{h} \right)_{T} =& \sum_{T\in\mathcal{T}_{h}} \left( e_c - Ie_{c}, R_{T} \right)_{T}\\
    \leq& \sum_{T\in\mathcal{T}_{h}} \left\Vert R_{T} \right\Vert_{T} \left\Vert e_c - Ie_{c} \right\Vert_{T}\\
    \leq& \sum_{T\in\mathcal{T}_{h}} \alpha_{T} \left\Vert R_{T} \right\Vert_{T} \left\Vert e_{c} \right\Vert_{\varepsilon,T}.
  \end{align*}
  For the third term, applying the Cauchy-Schwarz inequality and the edge estimate (\ref{I_e0}) gives
  \begin{align*}
    \sum_{T\in\mathcal{T}_{h}} \left\langle \left(\nabla\cdot \varepsilon^{2} D_{w}^{2}u_{h} - \nabla_{w}u_{h}\right) \cdot \mathbf{n}, e_c-Ie_{c} \right\rangle_{\partial T} =& \sum_{e\in \mathcal{E}_{h}} \left\langle J_{e,1}, e_c-Ie_{c} \right\rangle_{e}\\
    \leq& \sum_{e\in \mathcal{E}_{h}} \left\Vert J_{e,1} \right\Vert_{e} \left\Vert e_c - Ie_c \right\Vert_{e}\\
    \leq& \sum_{e\in \mathcal{E}_{h}} \alpha_{e,1} \left\Vert J_{e,1} \right\Vert_{e} \left\Vert e_{c} \right\Vert_{\varepsilon,T}
  \end{align*}
  For the fourth term, using the Cauchy-Schwarz inequality and the interpolation property (\ref{I_e1}) leads to 
  \begin{align*}
    \sum_{T\in\mathcal{T}_{h}} \left\langle D_{w}^{2}u_{h} \mathbf{n}, \nabla e_c-\{\nabla Ie_{c}\} \right\rangle_{\partial T} =& \sum_{e\in \mathcal{E}_{h}} \left\langle J_{e,2}, \nabla e_c-\{\nabla Ie_{c}\} \right\rangle_{e}\\
    \leq& \sum_{e\in \mathcal{E}_{h}} \left\Vert J_{e,2} \right\Vert_{e} \left( \left\Vert \nabla (e_c- Ie_{c}^{+}) \right\Vert_{e} + \left\Vert \nabla (e_c- Ie_{c}^{-}) \right\Vert_{e} \right)\\
    \leq& \sum_{e\in \mathcal{E}_{h}} \alpha_{e,2} \left\Vert J_{e,2} \right\Vert_{e} \left\Vert e_{c} \right\Vert_{\varepsilon,T}
  \end{align*}
  Similarly, for the remaining two terms, applying Cauchy-Schwarz inequality and (\ref{I_e1}) yields
  \begin{align*}
    S_{1}\left(u_{h}, Ie_{c}\right) =& \sum_{T\in\mathcal{T}_h} \varepsilon^{2} h_{T}^{-1}\left\langle\nabla u_0-\mathbf{u}_g, \nabla Ie_c- \{\nabla Ie_c\}\right\rangle_{\partial T}\\
    \leq& \sum_{T\in\mathcal{T}_h} \varepsilon^{2} h_{T}^{-1} \left\Vert \nabla u_0-\mathbf{u}_g \right\Vert_{\partial T} \left( \left\Vert \nabla (Ie_c- Ie_c^{+}) \right\Vert_{\partial T} + \left\Vert \nabla (Ie_c- Ie_c^{-}) \right\Vert_{\partial T}\right)\\
    \leq& \sum_{T\in\mathcal{T}_h} \varepsilon h_{T}^{-1/2} \left\Vert \nabla u_0-\mathbf{u}_g \right\Vert_{\partial T} \varepsilon \left\vert e_c \right\vert_{2,T}\\
    \leq& S_{1}(u_h,u_h)^{1/2} \left\Vert e_c \right\Vert_{\varepsilon,T}, 
  \end{align*}
  and
  \begin{align*}
    S_{2}\left(u_{h}, Ie_{c}\right) =& \sum_{T\in\mathcal{T}_h} h_{T}\left\langle\nabla u_0-\mathbf{u}_g, \nabla Ie_c- \{\nabla Ie_c\}\right\rangle_{\partial T}\\
    \leq& \sum_{T\in\mathcal{T}_h} h_{T} \left\Vert \nabla u_0-\mathbf{u}_g \right\Vert_{\partial T} \left( \left\Vert \nabla (Ie_c- Ie_c^{+}) \right\Vert_{\partial T} + \left\Vert \nabla (Ie_c- Ie_c^{-}) \right\Vert_{\partial T}\right)\\
    \leq& \sum_{T\in\mathcal{T}_h} h_{T}^{1/2} \left\Vert \nabla u_0-\mathbf{u}_g \right\Vert_{\partial T} \left\vert e_c \right\vert_{1,T}\\
    \leq& S_{2}(u_h,u_h)^{1/2} \left\Vert e_c \right\Vert_{\varepsilon,T}.
  \end{align*}
  Combining all the above estimates, we conclude that
  \begin{align*}
    l_{1}(u_{h}) \leq& C \eta_{h} \Vert e_c \Vert_{\varepsilon}.
  \end{align*}
\end{proof}
\begin{lemma}\label{uplemma4.3}
  Let $u\in H^2_0(\Omega)$ and $u_h \in V_{h}^{0}$ be the solutions to (\ref{1.1})-(\ref{1.2}) and (\ref{3.3}), respectively. Then we have
  \begin{align}
    l_{2}(u_{h}) \leq& \frac{1}{4} \left( \varepsilon^{2} \left\Vert D^{2} u - D_{w}^{2}u_{h} \right\Vert_{\mathcal{T}_{h}}^{2} + \left\Vert \nabla u - \nabla_{w}u_{h} \right\Vert_{\mathcal{T}_{h}}^{2} \right) + C \eta_{h}^{2},
  \end{align}
  where
  \begin{align*}
    l_{2}(u_{h}) = \varepsilon^{2} \left( D^{2} u - D_{w}^{2}u_{h}, D_{w}^{2}e_{d} \right)_{\mathcal{T}_{h}} + \left( \nabla u - \nabla_{w}u_{h}, \nabla_{w}e_{d} \right)_{\mathcal{T}_{h}}.
  \end{align*}
\end{lemma}
\begin{proof}
  Applying the Cauchy-Schwarz and triangle inequalities gives
  \begin{align*}
    l_{2}(u_{h}) \leq& \sum_{T\in\mathcal{T}_{h}} \left( \varepsilon\left\Vert D^{2} u - D_{w}^{2}u_{h} \right\Vert_{T} \varepsilon\left\Vert D_{w}^{2}e_{d} \right\Vert_{T} + \left\Vert \nabla u - \nabla_{w}u_{h} \right\Vert_{T} \left\Vert \nabla_{w}e_{d} \right\Vert_{T} \right)\\
    \leq& \left( \varepsilon^{2} \left\Vert D^{2} u - D_{w}^{2}u_{h} \right\Vert_{\mathcal{T}_{h}}^{2} + \left\Vert \nabla u - \nabla_{w}u_{h} \right\Vert_{\mathcal{T}_{h}}^{2} \right)^{1/2} \left( \varepsilon^{2} \left\Vert D_{w}^{2}e_{d} \right\Vert_{\mathcal{T}_{h}}^{2} +  \left\Vert \nabla_{w}e_{d} \right\Vert_{\mathcal{T}_{h}}^{2} \right)^{1/2}.
  \end{align*}
  From definition (\ref{3.1}), we derive
  \begin{align*}
    \varepsilon^{2} \left\Vert D_{w}^{2}e_{d} \right\Vert_{T}^{2} =& \varepsilon^{2} \left( u_{c}-u_{0}, \nabla\cdot\nabla\cdot D_{w}^{2}e_{d} \right)_{T} - \varepsilon^{2}\left\langle u_{c}-u_{b}, \nabla\cdot D_{w}^{2}e_{d} \cdot\mathbf{n} \right\rangle_{\partial T} + \varepsilon^{2}\left\langle \nabla u_{c}- \mathbf{u}_{g}, D_{w}^{2}e_{d} \mathbf{n} \right\rangle_{\partial T}\\
    =& \varepsilon^{2} \left( D^{2}(u_{c}-u_{0}), D_{w}^{2}e_{d} \right)_{T} + \varepsilon^{2}\left\langle u_{c}-u_{0}, \nabla\cdot D_{w}^{2}e_{d} \cdot\mathbf{n} \right\rangle_{\partial T} - \varepsilon^{2}\left\langle \nabla u_{c}- \nabla u_{0}, D_{w}^{2}e_{d} \mathbf{n} \right\rangle_{\partial T}\\
    &- \varepsilon^{2}\left\langle u_{c}-u_{b}, \nabla\cdot D_{w}^{2}e_{d} \cdot\mathbf{n} \right\rangle_{\partial T} + \varepsilon^{2}\left\langle \nabla u_{c}- \mathbf{u}_{g}, D_{w}^{2}e_{d} \mathbf{n} \right\rangle_{\partial T}\\
    =& \varepsilon^{2} \left( D^{2}(u_{c}-u_{0}), D_{w}^{2}e_{d} \right)_{T} - \varepsilon^{2}\left\langle u_{0}-u_{b}, \nabla\cdot D_{w}^{2}e_{d} \cdot\mathbf{n} \right\rangle_{\partial T} + \varepsilon^{2}\left\langle \nabla u_{0}- \mathbf{u}_{g}, D_{w}^{2}e_{d} \mathbf{n} \right\rangle_{\partial T}.
  \end{align*}
  Thanks to the single-valuedness of $u_{b}$ and $\mathbf{u}_{g}$ over each edge $e$. By the definition of $u_{c}$ and Lemma \ref{ROP}, we obtain
  \begin{align}
    \sum_{T\in\mathcal{T}_{h}} \varepsilon^{2} \left\Vert D^{2}(u_{c}-u_{0}) \right\Vert_{T}^{2} \leq& C \sum_{e \in\mathcal{E}_{h}} \left( \varepsilon^{2} h_{e}^{-3} \left\Vert [u_{0}] \right\Vert_{e}^{2} + \varepsilon^{2} h_{e}^{-1} \left\Vert [\nabla u_{0}] \right\Vert_{e}^{2} \right) \notag \\
    \leq& C \sum_{e \in\mathcal{E}_{h}} \left( \varepsilon^{2} h_{e}^{-3} \left\Vert [u_{0}-u_{b}] \right\Vert_{e}^{2} + \varepsilon^{2} h_{e}^{-1} \left\Vert [\nabla u_{0}-\mathbf{u}_{g}] \right\Vert_{e}^{2} \right) \notag \\
    \leq& C \sum_{T \in\mathcal{T}_{h}} \left( \varepsilon^{2} h_{T}^{-3} \left\Vert u_{0}-u_{b} \right\Vert_{T}^{2} + \varepsilon^{2} h_{T}^{-1} \left\Vert \nabla u_{0}-\mathbf{u}_{g} \right\Vert_{T}^{2} \right) \notag \\
    \leq& C S_{1}(u_{h}, u_{h}).\label{up4.9}
  \end{align}
  Using the Cauchy-Schwarz inequality and combining with (\ref{up4.9}) yields
  \begin{align}
    \varepsilon^{2} \left\Vert D_{w}^{2}e_{d} \right\Vert_{\mathcal{T}_{h}}^{2} \leq& C \sum_{T\in\mathcal{T}_{h}} \left(\right. \varepsilon \left\Vert D^{2}(u_{c}-u_{0}) \right\Vert_{T} \varepsilon \left\Vert D_{w}^{2}e_{d} \right\Vert_{T} \notag \\ 
    &+ \varepsilon \left\Vert u_{0}-u_{b} \right\Vert_{\partial T}  \varepsilon \left\Vert \nabla\cdot D_{w}^{2}e_{d} \cdot\mathbf{n} \right\Vert_{\partial T} + \varepsilon \left\Vert \nabla u_{0}- \mathbf{u}_{g} \right\Vert_{\partial T} \varepsilon \left\Vert D_{w}^{2}e_{d} \mathbf{n} \right\Vert_{\partial T} \left.\right) \notag \\
    \leq& \left( \sum_{T\in\mathcal{T}_{h}} \left( \varepsilon^2 \left\Vert D^{2}(u_{c}-u_{0}) \right\Vert_{T}^{2} + \varepsilon^2 h_{T}^{-3} \left\Vert u_{0}-u_{b} \right\Vert_{\partial T}^{2} + \varepsilon^2 h_{T}^{-1} \left\Vert \nabla u_{0}- \mathbf{u}_{g} \right\Vert_{\partial T}^{2} \right) \right)^{1/2} \notag \\
    &\left( \varepsilon^2 \left\Vert D_{w}^{2}e_{d} \right\Vert_{\mathcal{T}_{h}}^{2} \right)^{1/2} \notag \\
    \leq& C S_{1}(u_{h}, u_{h})^{1/2} \left( \varepsilon^2 \left\Vert D_{w}^{2}e_{d} \right\Vert_{\mathcal{T}_{h}}^{2} \right)^{1/2}.\label{up4.10}
  \end{align}

  Similarly, from definition (\ref{3.2}) we have
  \begin{align*}
    \left( \nabla_{w}e_{d}, \nabla_{w}e_{d} \right)_{T} =& - \left( u_{c}-u_{0}, \nabla\cdot \nabla_{w}e_{d} \right)_{T} + \left\langle u_{c}-u_{b}, \nabla_{w}e_{d} \cdot\mathbf{n} \right\rangle_{\partial T}\\
    =& \left( \nabla(u_{c}-u_{0}), \nabla_{w}e_{d} \right)_{T} - \left\langle u_{c}-u_{0}, \nabla_{w}e_{d} \cdot\mathbf{n} \right\rangle_{\partial T} + \left\langle u_{c}-u_{b}, \nabla_{w}e_{d} \cdot\mathbf{n} \right\rangle_{\partial T}\\
    =& \left( \nabla(u_{c}-u_{0}), \nabla_{w}e_{d} \right)_{T} + \left\langle u_{0}-u_{b}, \nabla_{w}e_{d} \cdot\mathbf{n} \right\rangle_{\partial T}.
  \end{align*}
  By the definition of $u_{c}$ and Lemma \ref{ROP}, and using the condition single-valuedness of $u_{b}$, we get
  \begin{align}
    \sum_{T\in\mathcal{T}_{h}} \left\Vert \nabla(u_{c}-u_{0}) \right\Vert_{T}^{2} \leq& C \sum_{e \in\mathcal{E}_{h}} \left( h_{e}^{-1} \left\Vert [u_{0}] \right\Vert_{e}^{2} + h_{e} \left\Vert [\nabla u_{0}] \right\Vert_{e}^{2} \right) \notag \\
    \leq& C \sum_{e \in\mathcal{E}_{h}} \left( h_{e}^{-1} \left\Vert [u_{0}-u_{b}] \right\Vert_{e}^{2} + h_{e} \left\Vert [\nabla u_{0}-\mathbf{u}_{g}] \right\Vert_{e}^{2} \right) \notag \\
    \leq& C \sum_{T \in\mathcal{T}_{h}} \left( h_{T}^{-1} \left\Vert u_{0}-u_{b} \right\Vert_{T}^{2} + h_{T} \left\Vert \nabla u_{0}-\mathbf{u}_{g} \right\Vert_{T}^{2} \right) \notag \\
    \leq& C S_{2}(u_{h}, u_{h}).\label{up4.11}
  \end{align}
  Applying the Cauchy-Schwarz inequality and combining with (\ref{up4.11}) gives
  \begin{align}
    \left\Vert \nabla_{w}e_{d} \right\Vert_{\mathcal{T}_{h}}^{2} \leq& C \sum_{T\in\mathcal{T}_{h}} \left( \left\Vert \nabla(u_{c}-u_{0}) \right\Vert_{T} \left\Vert \nabla_{w}e_{d} \right\Vert_{T} + \left\Vert u_{0}-u_{b} \right\Vert_{\partial T} \left\Vert \nabla_{w}e_{d} \cdot\mathbf{n} \right\Vert_{\partial T} \right) \notag \\
    \leq& \left( \sum_{T\in\mathcal{T}_{h}} \left( \left\Vert \nabla(u_{c}-u_{0}) \right\Vert_{T}^{2} + h_{T}^{-1} \left\Vert u_{0}-u_{b} \right\Vert_{\partial T}^{2} \right) \right)^{1/2} \left( \left\Vert \nabla_{w}e_{d} \right\Vert_{\mathcal{T}_{h}}^{2} \right)^{1/2} \notag \\
    \leq& C S_{2}(u_{h}, u_{h})^{1/2} \left( \left\Vert \nabla_{w}e_{d} \right\Vert_{\mathcal{T}_{h}}^{2} \right)^{1/2}.\label{up4.12}
  \end{align}
  Adding inequalities (\ref{up4.10}) and (\ref{up4.12}) yields
  \begin{align*}
    \varepsilon^{2} \left\Vert D_{w}^{2}e_{d} \right\Vert_{\mathcal{T}_{h}}^{2} + \left\Vert \nabla_{w}e_{d} \right\Vert_{\mathcal{T}_{h}}^{2} \leq& C \left( S_{1}(u_{h}, u_{h})^{1/2} \left( \varepsilon^2 \left\Vert D_{w}^{2}e_{d} \right\Vert_{\mathcal{T}_{h}}^{2} \right)^{1/2} + S_{2}(u_{h}, u_{h})^{1/2} \left( \left\Vert \nabla_{w}e_{d} \right\Vert_{\mathcal{T}_{h}}^{2} \right)^{1/2} \right)\\
    \leq& C \left( S_{1}(u_{h}, u_{h})  + S_{2}(u_{h}, u_{h}) \right)^{1/2} \left( \varepsilon^2 \left\Vert D_{w}^{2}e_{d} \right\Vert_{\mathcal{T}_{h}}^{2} + \left\Vert \nabla_{w}e_{d} \right\Vert_{\mathcal{T}_{h}}^{2} \right)^{1/2}\\
    \leq& C \eta_{h} \left( \varepsilon^2 \left\Vert D_{w}^{2}e_{d} \right\Vert_{\mathcal{T}_{h}}^{2} + \left\Vert \nabla_{w}e_{d} \right\Vert_{\mathcal{T}_{h}}^{2} \right)^{1/2},
  \end{align*}
  which implies 
  \begin{align}\label{up4.13}
    \left( \varepsilon^2 \left\Vert D_{w}^{2}e_{d} \right\Vert_{\mathcal{T}_{h}}^{2} + \left\Vert \nabla_{w}e_{d} \right\Vert_{\mathcal{T}_{h}}^{2} \right)^{1/2} \leq C \eta_{h}.
  \end{align}
  Combining the above results, we conclude
  \begin{align*}
    l_{2}(u_{h}) \leq& C \eta_{h} \left( \varepsilon^{2} \left\Vert D^{2} u - D_{w}^{2}u_{h} \right\Vert_{\mathcal{T}_{h}}^{2} + \left\Vert \nabla u - \nabla_{w}u_{h} \right\Vert_{\mathcal{T}_{h}}^{2} \right)^{1/2}\\
    \leq& \frac{1}{4} \left( \varepsilon^{2} \left\Vert D^{2} u - D_{w}^{2}u_{h} \right\Vert_{\mathcal{T}_{h}}^{2} + \left\Vert \nabla u - \nabla_{w}u_{h} \right\Vert_{\mathcal{T}_{h}}^{2} \right) + C \eta_{h}^{2}.
  \end{align*}
\end{proof}

\begin{lemma}\label{uplemma4.4}
  Since $e_{c}\in H_{0}^{2}(\Omega)$, the following estimate holds
  \begin{align}
    \Vert e_c \Vert_{\varepsilon} \leq& C \left( \varepsilon^{2} \left\Vert D^{2} u - D_{w}^{2} u_{h} \right\Vert_{\mathcal{T}_{h}}^{2} + \left\Vert \nabla u - \nabla_{w} u_{c} \right\Vert_{\mathcal{T}_{h}}^{2} \right)^{1/2} +C \eta_{h}.
  \end{align}
\end{lemma}
\begin{proof}
  By the definition of the $\Vert \cdot \Vert_{\varepsilon}$ and the previously established inequality (\ref{up4.13}), we derive the following estimate
  \begin{align*}
    \Vert e_c \Vert_{\varepsilon}^{2} =& \sum_{T\in\mathcal{T}_{h}} \left( \varepsilon^{2} \left\Vert D^{2} (u - u_{c}) \right\Vert_{T}^{2} + \left\Vert \nabla (u - u_{c}) \right\Vert_{T}^{2} \right)\\
    \leq& C \sum_{T\in\mathcal{T}_{h}} \left( \varepsilon^{2} \left\Vert D^{2} u - D_{w}^{2} u_{h} \right\Vert_{T}^{2} + \left\Vert \nabla u - \nabla_{w} u_{c} \right\Vert_{T}^{2} +\varepsilon^{2} \left\Vert D_{w}^{2} e_{d} \right\Vert_{T}^{2} + \left\Vert \nabla_{w} e_{d} \right\Vert_{T}^{2} \right)\\
    \leq& C \sum_{T\in\mathcal{T}_{h}} \left( \varepsilon^{2} \left\Vert D^{2} u - D_{w}^{2} u_{h} \right\Vert_{T}^{2} + \left\Vert \nabla u - \nabla_{w} u_{c} \right\Vert_{T}^{2} \right) + C \eta_{h}^2,
  \end{align*}
  which implies
  \begin{align*}
    \Vert e_c \Vert_{\varepsilon} \leq& C \left( \varepsilon^{2} \left\Vert D^{2} u - D_{w}^{2} u_{h} \right\Vert_{\mathcal{T}_{h}}^{2} + \left\Vert \nabla u - \nabla_{w} u_{c} \right\Vert_{\mathcal{T}_{h}}^{2} \right)^{1/2} +C \eta_{h}.
  \end{align*}
\end{proof}

\begin{theorem}
  Let $u\in H^2_0(\Omega)$ be the solution to (\ref{1.1})-(\ref{1.2}), and let $u_h \in V_{h}^{0}$ be the numerical approximation obtained from the scheme (\ref{3.3}). Then there exists a positive constant $C$, independent of $h$, $u$ and $u_h$, so that
  \begin{equation}\label{upper_bound}
    \3bar u-u_{h} \3bar \leq C \eta_{h}.
  \end{equation}
\end{theorem}
\begin{proof}
  For any cell $T\in \mathcal{T}_{h}$, applying Lemma \ref{lemma2.1} and Lemma \ref{lemma2.2} yields the following bounds 
  \begin{align*}
    &\left\Vert D_{w}^{2} (u-u_{h}) \right\Vert_{T} \leq \left\Vert D^{2} u - D_{w}^{2}u_{h} \right\Vert_{T},
    &&\left\Vert \nabla_{w} (u-u_{h}) \right\Vert_{T} \leq \left\Vert \nabla u - \nabla_{w}u_{h} \right\Vert_{T}.
  \end{align*}
  Consequently, we obtain
  \begin{align}\label{up4.16}
    \3bar u-u_{h} \3bar^{2} \leq& C \left( \varepsilon^{2} \left\Vert D^{2} u - D_{w}^{2}u_{h} \right\Vert_{\mathcal{T}_{h}}^{2} + \left\Vert \nabla u - \nabla_{w}u_{h} \right\Vert_{\mathcal{T}_{h}}^{2} \right) +S_{1}(u_{h}, u_{h}) + S_{2}(u_{h}, u_{h}).
  \end{align}

  We now consider the error term
  \begin{align*}
    \varepsilon^{2} \left\Vert D^{2} u - D_{w}^{2}u_{h} \right\Vert_{\mathcal{T}_{h}}^{2} + \left\Vert \nabla u - \nabla_{w}u_{h} \right\Vert_{\mathcal{T}_{h}}^{2} = l_{1}(u_{h})+l_{2}(u_{h}), 
  \end{align*}
  where
  \begin{align*}
    &l_{1}(u_{h}) = \varepsilon^{2} \left( D^{2} u - D_{w}^{2}u_{h}, D^{2} u - D_{w}^{2}u_c \right)_{\mathcal{T}_{h}} + \left( \nabla u - \nabla_{w}u_{h}, \nabla u - \nabla_{w}u_c \right)_{\mathcal{T}_{h}},\\
    &l_{2}(u_{h}) = \varepsilon^{2} \left( D^{2} u - D_{w}^{2}u_{h}, D_{w}^{2}e_{d} \right)_{\mathcal{T}_{h}} + \left( \nabla u - \nabla_{w}u_{h}, \nabla_{w}e_{d} \right)_{\mathcal{T}_{h}}.
  \end{align*}
  By Lemma \ref{uplemma4.2} and Lemma \ref{uplemma4.4}, we have
  \begin{align}
    l_{1}(u_{h}) \leq& C \eta_{h} \Vert e_c \Vert_{\varepsilon} \notag \\
    \leq& C\eta_{h} \left( \varepsilon^{2} \left\Vert D^{2} u - D_{w}^{2} u_{h} \right\Vert_{\mathcal{T}_{h}}^{2} + \left\Vert \nabla u - \nabla_{w} u_{c} \right\Vert_{\mathcal{T}_{h}}^{2} \right)^{1/2} + C \eta_{h}^{2} \notag \\
    \leq& \frac{1}{4} \left( \varepsilon^{2} \left\Vert D^{2} u - D_{w}^{2} u_{h} \right\Vert_{\mathcal{T}_{h}}^{2} + \left\Vert \nabla u - \nabla_{w} u_{c} \right\Vert_{\mathcal{T}_{h}}^{2} \right) + C \eta_{h}^{2}.\label{up4.17}
  \end{align}
  Combining (\ref{up4.17}) with Lemma \ref{uplemma4.3}, we obtain 
  \begin{align*}
    \varepsilon^{2} \left\Vert D^{2} u - D_{w}^{2}u_{h} \right\Vert_{\mathcal{T}_{h}}^{2} + \left\Vert \nabla u - \nabla_{w}u_{h} \right\Vert_{\mathcal{T}_{h}}^{2}
    \leq& \frac{1}{2} \left( \varepsilon^{2} \left\Vert D^{2} u - D_{w}^{2} u_{h} \right\Vert_{\mathcal{T}_{h}}^{2} + \left\Vert \nabla u - \nabla_{w} u_{c} \right\Vert_{\mathcal{T}_{h}}^{2} \right) + C \eta_{h}^{2}.
  \end{align*}
  which implies
  \begin{align}\label{up4.18}
    \varepsilon^{2} \left\Vert D^{2} u - D_{w}^{2}u_{h} \right\Vert_{\mathcal{T}_{h}}^{2} + \left\Vert \nabla u - \nabla_{w}u_{h} \right\Vert_{\mathcal{T}_{h}}^{2}
    \leq& C \eta_{h}^{2}.
  \end{align}
  Substituting (\ref{up4.18}) into (\ref{up4.16}) yields
  \begin{align*}
    \3bar u-u_{h} \3bar^{2} \leq& C \eta_{h}^{2},
  \end{align*}
  which completes the proof.
\end{proof}

\subsection{Lower bound}
In this section, we establish the efficiency of the proposed a posteriori error estimator for guiding adaptive mesh refinement in the singularly perturbed problem. To derive the efficiency bounds, we employ bubble function techniques.

Let $b_T: T \to \mathbb{R}$ denote the standard interior bubble function on a cell $T$, defined by $b_T = b_{\widehat{T}} \circ F_T$, where $b_{\widehat{T}}$ is the reference bubble function. Specifically, if $\widehat{T}$ is the reference triangle with barycentric coordinates $\lambda_1, \lambda_2, \lambda_3$, then $b_{\widehat{T}} = 27\lambda_1\lambda_2\lambda_3$; if $\widehat{T}$ is the reference rectangle with coordinates $\lambda_1, \lambda_2$, then $b_{\widehat{T}} = (1 - \lambda_1^2)(1 - \lambda_2^2)$.

For each interior edge $e \in \mathcal{E}_h$, let $\widetilde{T} \subset T_{1} \cup T_{2}$ be the largest rhombus contained in the union of the two adjacent cells $T_{1}$ and $T_{2}$, with $e$ as one of its diagonals (see Fig. 5). We define $b_{\widetilde{T}}: \widetilde{T} \to \mathbb{R}$ as the corresponding bubble function on the rhombus $\widetilde{T}$.

The following theorem shows the efficiency of the estimator globally, which is a direct consequence of the last theorem.
\begin{theorem}
  Let $u\in H^2_0(\Omega)$ be the solution to (\ref{1.1})-(\ref{1.2}), and let $u_h \in V_{h}^{0}$ be the numerical approximation obtained from the scheme (\ref{3.3}). Then there exists a positive constant $C$, independent of $h$, $u$ and $u_h$, so that
  \begin{equation}\label{lower_bound}
    \eta_{h} \leq C \left( \3bar u - u_{h} \3bar + \left(\sum_{T\in\mathcal{T}_{h}} \alpha_{T}^{2} \left\Vert f-f_{h} \right\Vert_{T}^{2} \right)^{1/2}\right).
  \end{equation}
\end{theorem}
\begin{proof}
  Consider a fixed cell $T \in \mathcal{T}_{h}$ and let $v\in H_{0}^{2}(\Omega) \cap H_{0}^{2}(T)$ be a polynomial function on $T$ that vanishes on $\Omega\backslash T$. Applying Lemma \ref{lemma2.1}, Lemma \ref{lemma2.2}, and integration by parts yields  
  \begin{align}
    \left( R_{h}, v \right)_{T}
    =& \varepsilon^{2} \left( D^{2}u, D^{2}v \right)_{T} + \left( \nabla u, \nabla v \right)_{T} - \varepsilon^{2} \left( D_{w}^{2}u_{h}, D^{2}v \right)_{T} - \left( \nabla_{w} u_{h}, \nabla v \right)_{T} \notag \\
    &-  \left\langle \left(\nabla\cdot \varepsilon^{2} D_{w}^{2}u_{h} - \nabla_{w} u_{h} \right) \cdot\mathbf{n}, v \right\rangle_{\partial T} + \left\langle \varepsilon^{2} D_{w}^{2}u_{h} \mathbf{n}, \nabla v \right\rangle_{\partial T} - \left( f-f_{h}, v \right)_{T} \notag \\
    =& \varepsilon^{2} \left( D_{w}^{2}e_{h}, D^{2}v \right)_{T} + \left( \nabla_{w} e_{h}, \nabla v \right)_{T} - \left( f-f_{h}, v \right)_{T} \notag \\
    &-  \left\langle \left(\nabla\cdot \varepsilon^{2} D_{w}^{2}u_{h} - \nabla_{w} u_{h} \right) \cdot\mathbf{n}, v \right\rangle_{\partial T} + \left\langle \varepsilon^{2} D_{w}^{2}u_{h} \mathbf{n}, \nabla v \right\rangle_{\partial T}.\label{low4.20}
  \end{align}
  Now, set $v=b_{T}^{2} R_{T}$ in \eqref{low4.20}. Using the Cauchy-Schwarz inequality and inverse inequality, we obtain 
  \begin{align*}
    \left( R_{h}, b_{T}^{2} R_{T} \right)_{T} =& \varepsilon^{2} \left( D_{w}^{2}e_{h}, D^{2} b_{T}^{2} R_{T} \right)_{T} + \left( \nabla_{w} e_{h}, \nabla b_{T}^{2} R_{T} \right)_{T} - \left( f-f_{h}, b_{T}^{2} R_{T} \right)_{T}\\
    \leq& \varepsilon^{2} \left\Vert D_{w}^{2}e_{h} \right\Vert_{T} \left\Vert D^{2} b_{T}^{2} R_{T} \right\Vert_{T} + \left\Vert \nabla_{w}e_{h} \right\Vert_{T} \left\Vert \nabla b_{T}^{2} R_{T} \right\Vert_{T} + \left\Vert f - f_{h} \right\Vert_{T} \left\Vert b_{T}^{2} R_{T} \right\Vert_{T}\\
    \leq& \varepsilon \left\Vert D_{w}^{2}e_{h} \right\Vert_{T} \varepsilon h_{T}^{-2} \left\Vert b_{T}^{2} R_{T} \right\Vert_{T} + \left\Vert \nabla_{w}e_{h} \right\Vert_{T} h_{T}^{-1} \left\Vert b_{T}^{2} R_{T} \right\Vert_{T} + \alpha_{T} \left\Vert f - f_{h} \right\Vert_{T} \alpha_{T}^{-1} \left\Vert b_{T}^{2} R_{T} \right\Vert_{T}\\
    \leq& \left( \varepsilon \left\Vert D_{w}^{2}e_{h} \right\Vert_{T} + \left\Vert \nabla_{w}e_{h} \right\Vert_{T} + \alpha_{T} \left\Vert f - f_{h} \right\Vert_{T} \right) \alpha_{T}^{-1} \left\Vert b_{T}^{2} R_{T} \right\Vert_{T}\\
    \leq& \left( \varepsilon \left\Vert D_{w}^{2}e_{h} \right\Vert_{T} + \left\Vert \nabla_{w}e_{h} \right\Vert_{T} + \alpha_{T} \left\Vert f - f_{h} \right\Vert_{T} \right) \alpha_{T}^{-1} \left\Vert b_{T}^{2} R_{T} \right\Vert_{T}
  \end{align*}
  We note that the norm $\left\Vert \cdot b_{T} \right\Vert_{T}$ defines a norm on the finite-dimensional space $\mathbb{P}_{k+2}(T)$, and is therefore equivalent to the standard $L^{2}$-norm $\left\Vert\cdot\right\Vert_{T}$ on this space. In particular, we have
  \begin{align*}
    \left\Vert R_{T} \right\Vert_{T}^{2} \leq C \left( R_{T}, b_{T}^{2}R_{T} \right)_{T} \leq C \left( \varepsilon \left\Vert D_{w}^{2}e_{h} \right\Vert_{T} + \left\Vert \nabla_{w}e_{h} \right\Vert_{T} + \alpha_{T} \left\Vert f - f_{h} \right\Vert_{T} \right) \alpha_{T}^{-1} \left\Vert R_{T} \right\Vert_{T},
  \end{align*}
  which implies
  \begin{align}\label{low4.21}
    \alpha_{T} \left\Vert R_{T} \right\Vert_{T} \leq C \left( \varepsilon \left\Vert D_{w}^{2}e_{h} \right\Vert_{T} + \left\Vert \nabla_{w}e_{h} \right\Vert_{T} + \alpha_{T} \left\Vert f - f_{h} \right\Vert_{T} \right).
  \end{align}

  Assume $\phi$ is constant in the normal direction to the edge $e$. Let $l: e \to \mathbb{R}$ be defined such that $l(s)$ denotes the length of the intersection between the line normal to $e$ at point $s \in e$ and the domain $\widetilde{T}$. Then the following norm estimate holds:
  \begin{align*}
    \|\phi\|_{T_{1} \cup T_{2}} = \left( \int_e \phi^2(s) l(s)  ds \right)^{1/2} \leq C h_{e}^{1/2} \|\phi\|_e \leq C h_{T}^{1/2} \|\phi\|_e. 
  \end{align*}
  Let $b_l: \widetilde{T} \to \mathbb{R}$ be a linear polynomial that vanishes along the edge $e$, and whose gradient satisfies $\nabla b_l|_e = h_{e}^{-1} \mathbf{n}$. Using this, we define a function $b_e: \Omega \to \mathbb{R}$ by $b_{e}|_{\widetilde{T}} = b_{l} b_{\widetilde{T}}^3$. This function satisfies $b_{e} \in C(\Omega) \cap H_{0}^{2}(\Omega)$, and clearly $b_e = 0$ on $\Omega \setminus \widetilde{T}$ and $e$.
  
  Let $v=b_{e} J_{e,2} \cdot \mathbf{n}$ and substitute it into equation (\ref{low4.20}) over the domain $\cup_{T} = T_{1} \cup T_{2}$, Applying Cauchy-Schwarz inequality, inverse inequality and (\ref{low4.21}) yields 
  \begin{align*}
    \left\langle J_{e,2}, \nabla \left( b_{e} J_{e,2} \cdot \mathbf{n} \right) \right\rangle_{e}
    \leq& \varepsilon^{2} \left\Vert D_{w}^{2} e_{h} \right\Vert_{\cup_{T}} \left\Vert D^{2} \left( b_{e} J_{e,2} \cdot \mathbf{n} \right) \right\Vert_{\cup_{T}} + \left\Vert \nabla_{w} e_{h} \right\Vert_{\cup_{T}} \left\Vert \nabla \left( b_{e} J_{e,2} \cdot \mathbf{n} \right) \right\Vert_{\cup_{T}}\\
    &+ \left\Vert R_{T} \right\Vert_{\cup_{T}} \left\Vert b_{e} J_{e,2} \cdot \mathbf{n} \right\Vert_{\cup_{T}} + \left\Vert f-f_{h} \right\Vert_{\cup_{T}} \left\Vert b_{e} J_{e,2} \cdot \mathbf{n} \right\Vert_{\cup_{T}}\\
    \leq& C \Big( \varepsilon \left\Vert D_{w}^{2} e_{h} \right\Vert_{\cup_{T}} \varepsilon h_{T}^{-3/2} \left\Vert J_{e,2} \right\Vert_{e} + \left\Vert \nabla_{w} e_{h} \right\Vert_{\cup_{T}} h_{T}^{-1/2} \left\Vert J_{e,2} \right\Vert_{e}\\
    &+ \alpha_{T} \left\Vert R_{T} \right\Vert_{\cup_{T}} \alpha_{T}^{-1}h_{T}^{1/2} \left\Vert J_{e,2} \right\Vert_{e} + \alpha_{T} \left\Vert f-f_{h} \right\Vert_{\cup_{T}} \alpha_{T}^{-1}h_{T}^{1/2} \left\Vert J_{e,2} \right\Vert_{e} \Big)\\
    \leq& C \left( \varepsilon \left\Vert D_{w}^{2} e_{h} \right\Vert_{\cup_{T}} + \left\Vert \nabla_{w} e_{h} \right\Vert_{T} + \alpha_{T} \left\Vert R_{T} \right\Vert_{\cup_{T}} + \alpha_{T} \left\Vert f-f_{h} \right\Vert_{\cup_{T}} \right) \alpha_{e,1}^{-1} \left\Vert J_{e,2} \right\Vert_{e}\\
    \leq& C \left( \varepsilon \left\Vert D_{w}^{2} e_{h} \right\Vert_{\cup_{T}} + \left\Vert \nabla_{w} e_{h} \right\Vert_{\cup_{T}} + \alpha_{T} \left\Vert f-f_{h} \right\Vert_{\cup_{T}} \right) \alpha_{e,1}^{-1} \left\Vert J_{e,2} \right\Vert_{e},
  \end{align*}
  where we have used $\left\Vert b_{e}J_{e,2}\cdot \mathbf{n} \right\Vert_{\cup_{T}} \leq C\left\Vert J_{e,2} \right\Vert_{\cup_{T}} \leq C h_{T}^{1/2} \left\Vert J_{e,2} \right\Vert_{e}$.
  It can be directly verified that $\nabla \left( b_{e} J_{e,2} \cdot \mathbf{n} \right)|_{e} = h_{e}^{-1}\mathbf{n} b_{\widetilde{T}}^{3}|_{e} \left(j_{e,2}\cdot \mathbf{n}\right)|_{e}$. Consequently, we derive $\left\langle J_{e,2}, \nabla \left( b_{e} J_{e,2} \cdot \mathbf{n} \right) \right\rangle_{e} = h_{e}^{-1} \left\Vert  b_{\widetilde{T}}^{3/2} J_{e,2} \right\Vert_{e}^{2}$. 
  By norm equivalence and a scaling argument, we obtain the bound 
  \begin{align*}
    h_{T}^{-1}\left\Vert J_{e,2} \right\Vert_{e}^{2} \leq C h_{e}^{-1} \left\Vert  b_{\widetilde{T}}^{3/2} J_{e,2} \right\Vert_{e}^{2}
    \leq& C \left( \varepsilon \left\Vert D_{w}^{2} e_{h} \right\Vert_{\cup_{T}} + \left\Vert \nabla_{w} e_{h} \right\Vert_{\cup_{T}} + \alpha_{T} \left\Vert f-f_{h} \right\Vert_{\cup_{T}} \right) \alpha_{e,1}^{-1} \left\Vert J_{e,2} \right\Vert_{e},
  \end{align*}
  which implies
  \begin{align}\label{low4.22}
    \alpha_{e,2} \left\Vert J_{e,2} \right\Vert_{e} \leq C \left( \varepsilon \left\Vert D_{w}^{2} e_{h} \right\Vert_{\cup_{T}} + \left\Vert \nabla_{w} e_{h} \right\Vert_{\cup_{T}} + \alpha_{T} \left\Vert f-f_{h} \right\Vert_{\cup_{T}} \right),
  \end{align}
  since $\alpha_{e,2}=h_{T}^{-1} \alpha_{e,1}$.

  We now set $v=b_{\widetilde{T}}^{3} J_{e,1}$ and substitute it into equation (\ref{low4.20}) over the domain $\cup_{T} = T_{1} \cup T_{2}$, Applying Cauchy-Schwarz inequality, inverse inequality and (\ref{low4.21}), we obtain 
  \begin{align*}
    \left\langle J_{e,1}, b_{\widetilde{T}}^{3} J_{e,1} \right\rangle_{e}
    \leq& \varepsilon^{2} \left\Vert D_{w}^{2} e_{h} \right\Vert_{\cup_{T}} \left\Vert D^{2} \left(b_{\widetilde{T}}^{3} J_{e,1}\right) \right\Vert_{\cup_{T}} + \left\Vert \nabla_{w} e_{h} \right\Vert_{\cup_{T}} \left\Vert \nabla \left(b_{\widetilde{T}}^{3} J_{e,1}\right) \right\Vert_{\cup_{T}}\\
    &+ \left\Vert R_{T} \right\Vert_{\cup_{T}} \left\Vert b_{\widetilde{T}}^{3} J_{e,1} \right\Vert_{\cup_{T}} + \left\Vert f-f_{h} \right\Vert_{\cup_{T}} \left\Vert b_{\widetilde{T}}^{3} J_{e,1} \right\Vert_{\cup_{T}} + \left\Vert J_{e,2} \right\Vert_{e} \left\Vert \nabla \left(b_{\widetilde{T}}^{3} J_{e,1}\right) \right\Vert_{e}\\
    \leq& C\Big( \varepsilon \left\Vert D_{w}^{2} e_{h} \right\Vert_{\cup_{T}} \varepsilon h_{T}^{-3/2} \left\Vert J_{e,1} \right\Vert_{e} + \left\Vert \nabla_{w} e_{h} \right\Vert_{\cup_{T}} h_{T}^{-1/2} \left\Vert J_{e,1} \right\Vert_{e} + \alpha_{T} \left\Vert R_{T} \right\Vert_{\cup_{T}} \alpha_{T}^{-1}h_{T}^{1/2} \left\Vert J_{e,1} \right\Vert_{e}\\
    & + \alpha_{T} \left\Vert f-f_{h} \right\Vert_{\cup_{T}} \alpha_{T}^{-1}h_{T}^{1/2} \left\Vert J_{e,1} \right\Vert_{e} + \alpha_{e,2}\left\Vert J_{e,2} \right\Vert_{e} \alpha_{e,2}^{-1} h_{T}^{-1}\left\Vert J_{e,1} \right\Vert_{e} \Big)\\
    \leq& C \left( \varepsilon \left\Vert D_{w}^{2} e_{h} \right\Vert_{\cup_{T}} + \left\Vert \nabla_{w} e_{h} \right\Vert_{\cup_{T}} + \alpha_{T} \left\Vert R_{T} \right\Vert_{\cup_{T}} + \alpha_{T} \left\Vert f-f_{h} \right\Vert_{\cup_{T}} + \alpha_{e,2}\left\Vert J_{e,2} \right\Vert_{e}\right) \alpha_{e,1}^{-1} \left\Vert J_{e,1} \right\Vert_{e}\\
    \leq& C \left( \varepsilon \left\Vert D_{w}^{2} e_{h} \right\Vert_{\cup_{T}} + \left\Vert \nabla_{w} e_{h} \right\Vert_{\cup_{T}} + \alpha_{T} \left\Vert f-f_{h} \right\Vert_{\cup_{T}} \right) \alpha_{e,1}^{-1} \left\Vert J_{e,1} \right\Vert_{e},
  \end{align*}
  where we have used $\left\Vert b_{\widetilde{T}}^{3} J_{e,1} \right\Vert_{\cup_{T}} \leq C \left\Vert J_{e,1} \right\Vert_{\cup_{T}} \leq C h_{T}^{1/2} \left\Vert J_{e,2} \right\Vert_{e}$. 
  By norm equivalence and a scaling argument, we obtain the bound 
  \begin{align*}
    \left\Vert J_{e,1} \right\Vert_{e}^{2} = \left\langle J_{e,1}, b_{\widetilde{T}}^{3} J_{e,1} \right\rangle_{e} \leq& C \left( \varepsilon \left\Vert D_{w}^{2} e_{h} \right\Vert_{T} + \left\Vert \nabla_{w} e_{h} \right\Vert_{T} + \alpha_{T} \left\Vert f-f_{h} \right\Vert_{T} \right) \alpha_{e,1}^{-1} \left\Vert J_{e,1} \right\Vert_{e},
  \end{align*}
  it follows that
  \begin{align}\label{low4.23}
    \alpha_{e,1} \left\Vert J_{e,1} \right\Vert_{e} \leq& C \left( \varepsilon \left\Vert D_{w}^{2} e_{h} \right\Vert_{T} + \left\Vert \nabla_{w} e_{h} \right\Vert_{T} + \alpha_{T} \left\Vert f-f_{h} \right\Vert_{T} \right).
  \end{align}
  The desired result follows immediately from the definition of $\eta_{h}$, (\ref{low4.21}), (\ref{low4.22}) and (\ref{low4.23}).
\end{proof}

\section{Numerical Experiments}
In this section, we present a series of two-dimensional numerical experiments to assess the performance of the proposed a posteriori error estimator within an adaptive mesh refinement framework. Unless otherwise specified, we only consider $k=2$.
\begin{example}\label{exam6.1}
  Let $\Omega = (0, 1)^2$ and select the forcing function $f$ such that the exact solution of (\ref{1.1})-(\ref{1.2}) exhibits a sharp internal layer. The solution is given by
  \begin{align*}
    u(x, y)=xy(1-x)(1-y)exp\left(-1000\left((x-0.5)^{2}+(y-0.117)^{2}\right)\right).
  \end{align*}
\end{example}
We set the perturbation parameter to $\varepsilon = 1$ and $\theta = 0.3$. Figure \ref{Gc} illustrates the convergence history under adaptive refinement. The final adapted mesh is shown in Figure \ref{Gm}, while the exact and numerical solutions are displayed in Figures \ref{Gu} and \ref{Guh}, respectively. These results demonstrate that the adaptive strategy effectively refines the mesh near the singular region and that the error estimator agrees well with the error. 
\begin{figure}[h]
	\centering
	\subfigure[]{
		\begin{minipage}[h]{0.47\textwidth}
			\includegraphics[width=1\textwidth]{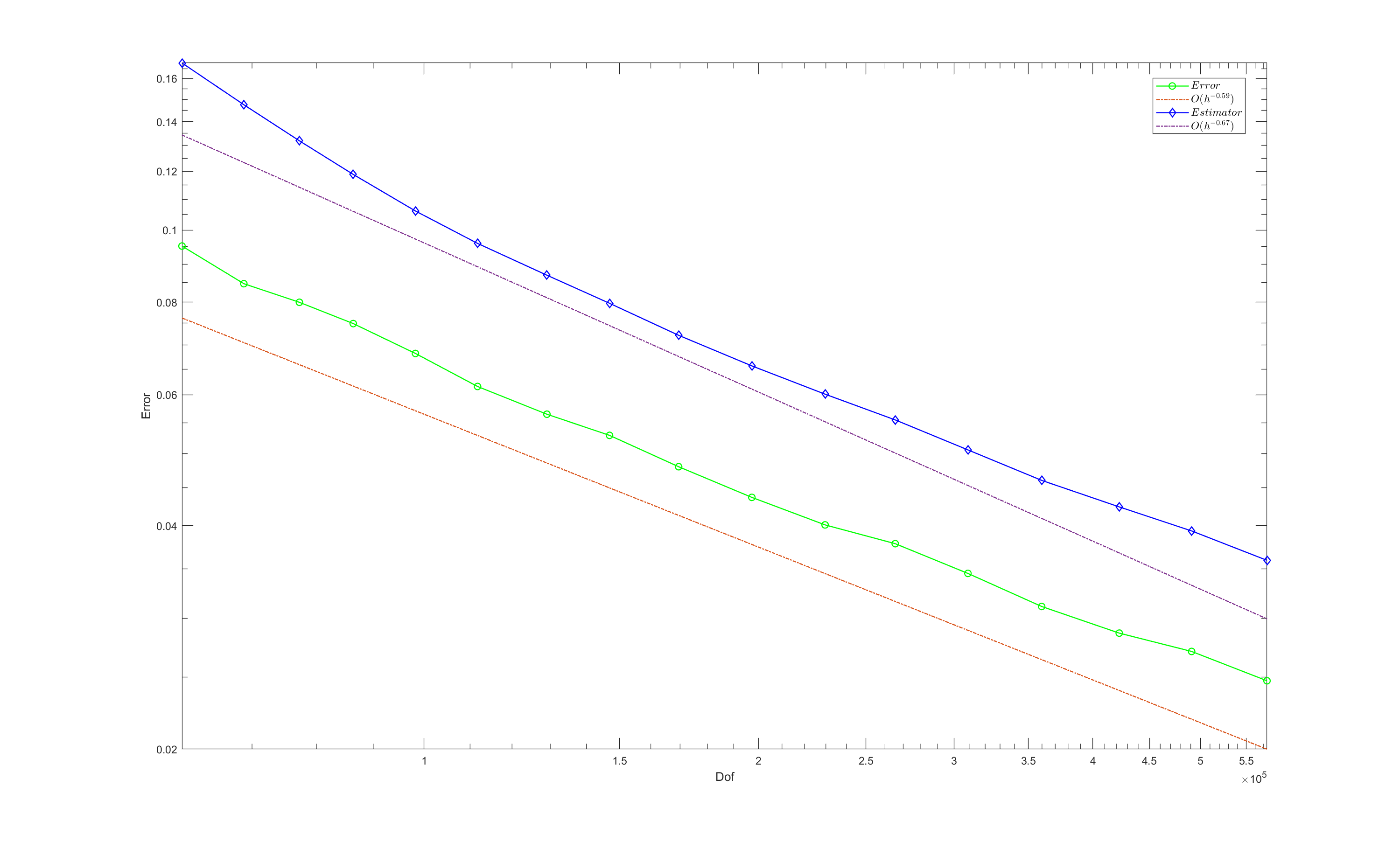} 
		\end{minipage}
		\label{Gc}
	}
    	\subfigure[]{
    		\begin{minipage}[h]{0.47\textwidth}
   		 	\includegraphics[width=1\textwidth]{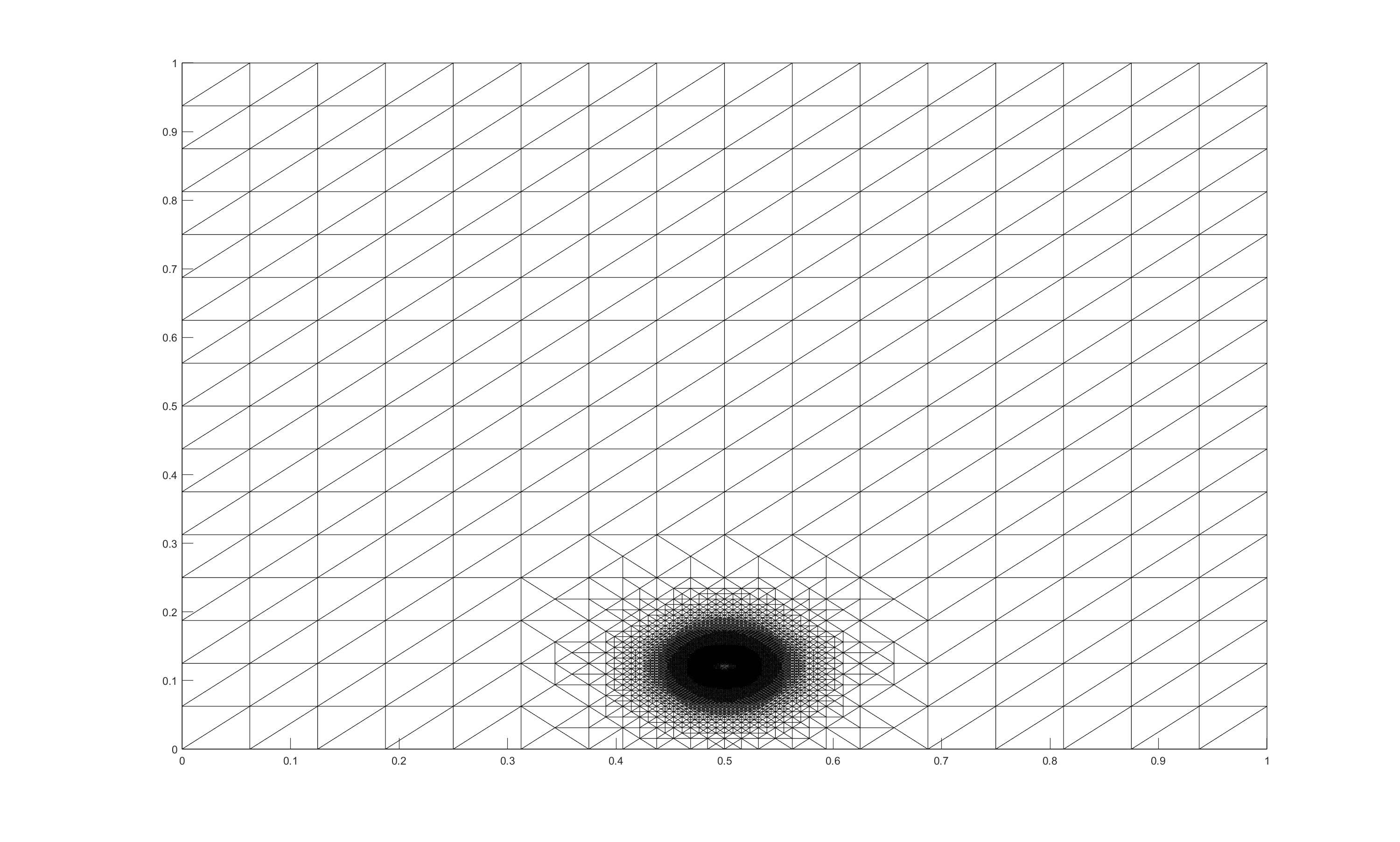}
    		\end{minipage}
		\label{Gm}
    	}
	\\ 
	\subfigure[]{
		\begin{minipage}[h]{0.48\textwidth}
			\includegraphics[width=1\textwidth]{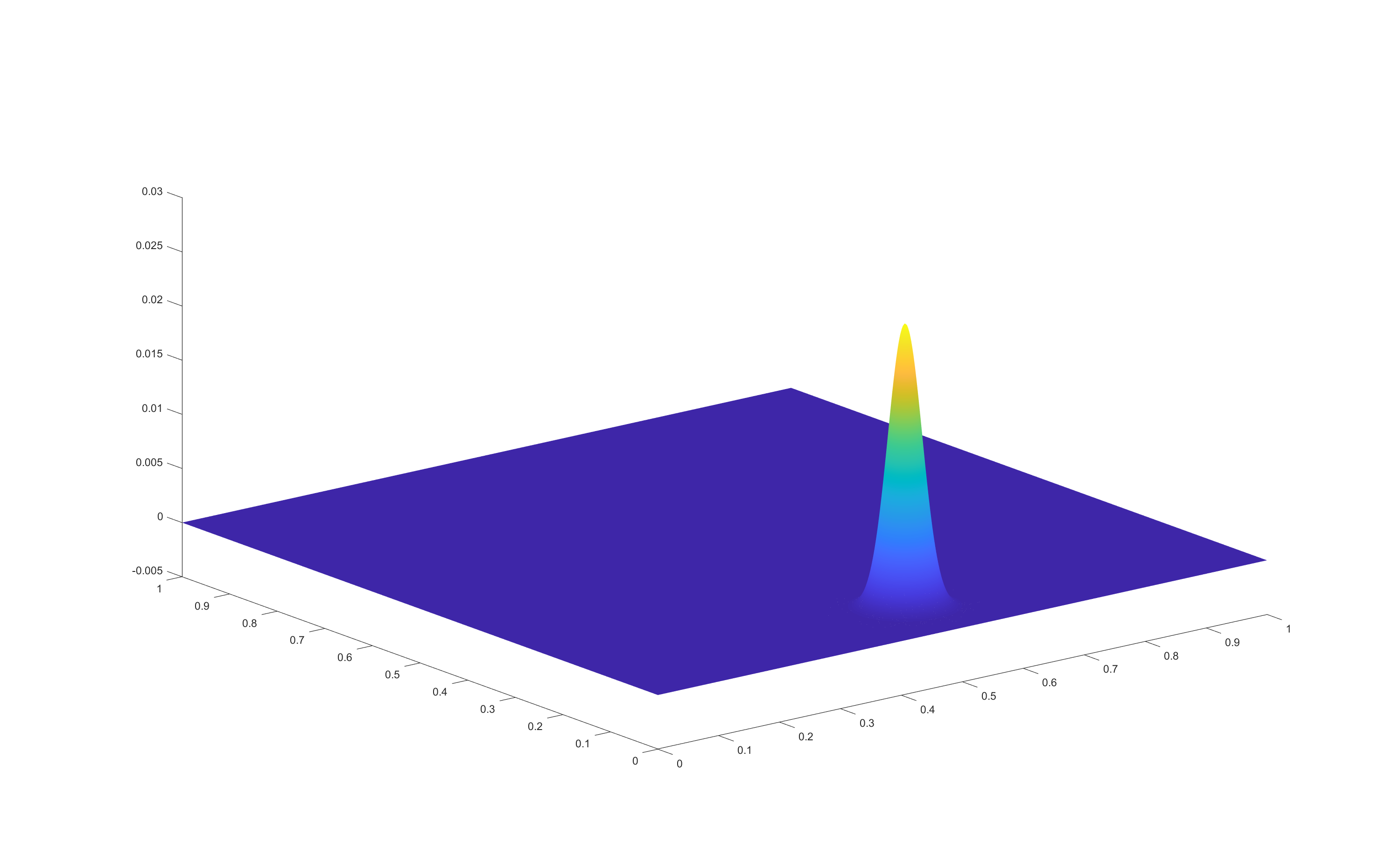} 
		\end{minipage}
		\label{Gu}
	}
    	\subfigure[]{
    		\begin{minipage}[h]{0.48\textwidth}
		 	\includegraphics[width=1\textwidth]{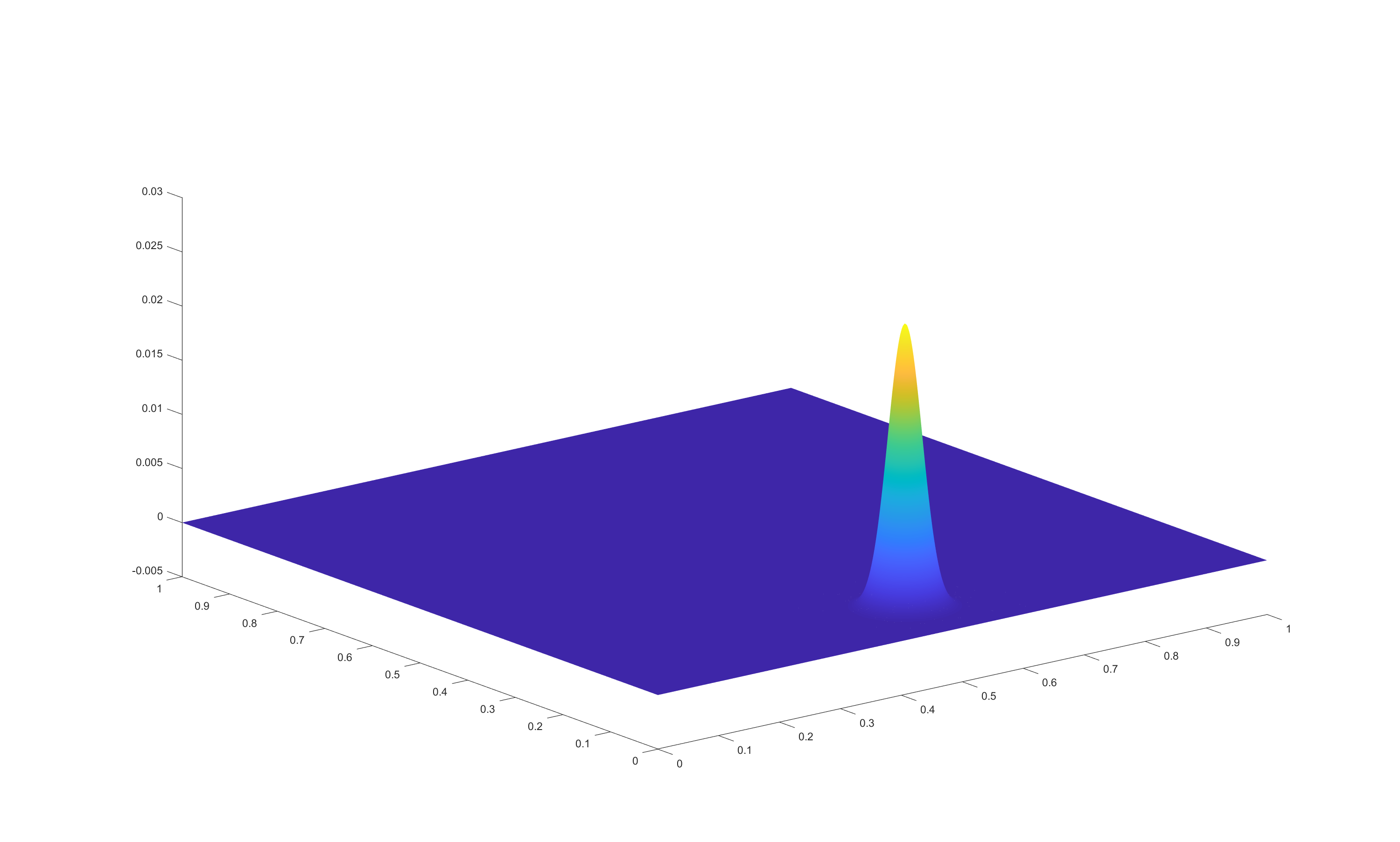}
    		\end{minipage}
		\label{Guh}
    	}
	\caption{(a) Convergence rates of the error and the error estimator; (b) The final adapted mesh; (c) Exact solution; (d) Numerical solution.}
	\label{fig1}
\end{figure}

\begin{example}\label{exam6.2}
  This example investigates the performance of the method in the presence of an interior layer. Let $\Omega = (0, 1)^2$ and select the source term $f$ such that the analytical solution to (\ref{1.1})-(\ref{1.2}) exhibits large gradients and is given by
  \begin{align*}
    u(x,y)=0.5x(1-x)(1-y)\left(1-\tanh \frac{\beta-x}{\gamma}\right). 
  \end{align*}
  Here, the parameters $\beta$ and $\gamma$ determine the location and thickness of the interior layer, respectively.
\end{example}
In this test, we set $\beta = 0.5$, $\gamma = 0.05$, $\varepsilon = 1$ and $\theta = 0.3$. Figure \ref{fig2} shows the convergence history under adaptive refinement, the final adapted mesh, and comparisons between the exact and numerical solutions. These results demonstrate that the adaptive scheme accurately captures the interior layer. 
\begin{figure}[h]
	\centering
	\subfigure[]{
		\begin{minipage}[h]{0.47\textwidth}
			\includegraphics[width=1\textwidth]{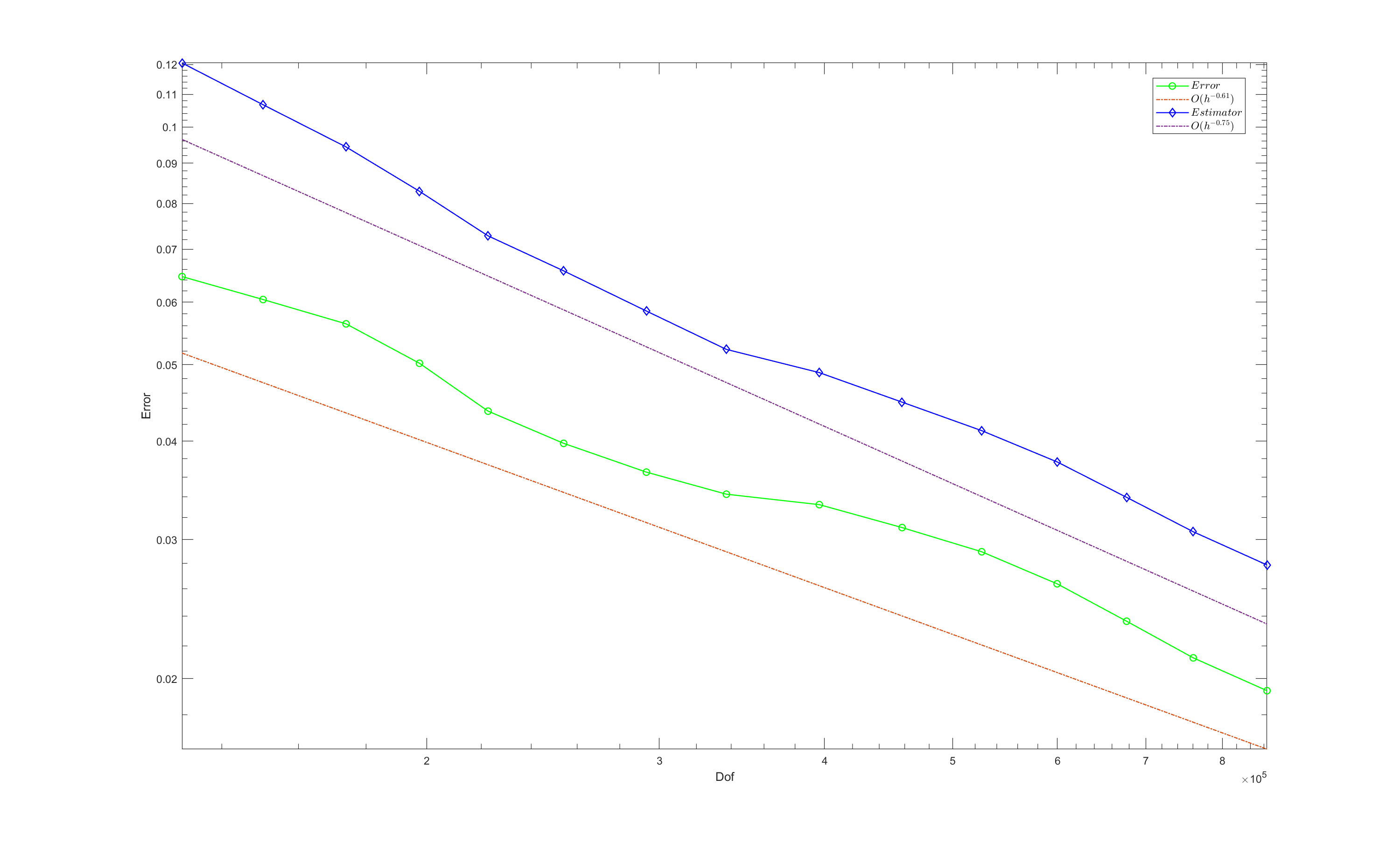} 
		\end{minipage}
		\label{Ic0.5}
	}
    	\subfigure[]{
    		\begin{minipage}[h]{0.47\textwidth}
   		 	\includegraphics[width=1\textwidth]{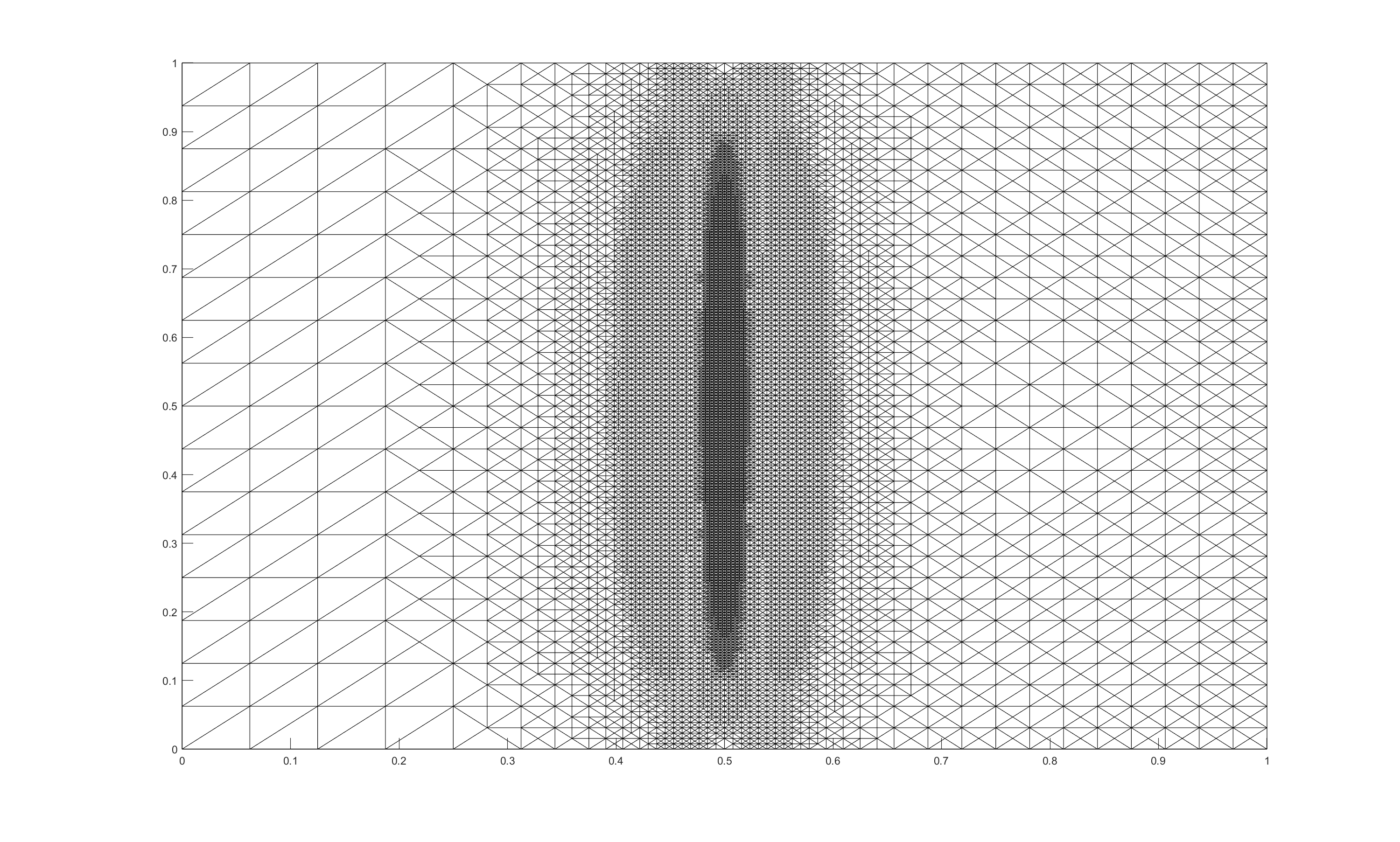}
    		\end{minipage}
		\label{Im0.5}
    	}
	\\ 
	\subfigure[]{
		\begin{minipage}[h]{0.48\textwidth}
			\includegraphics[width=1\textwidth]{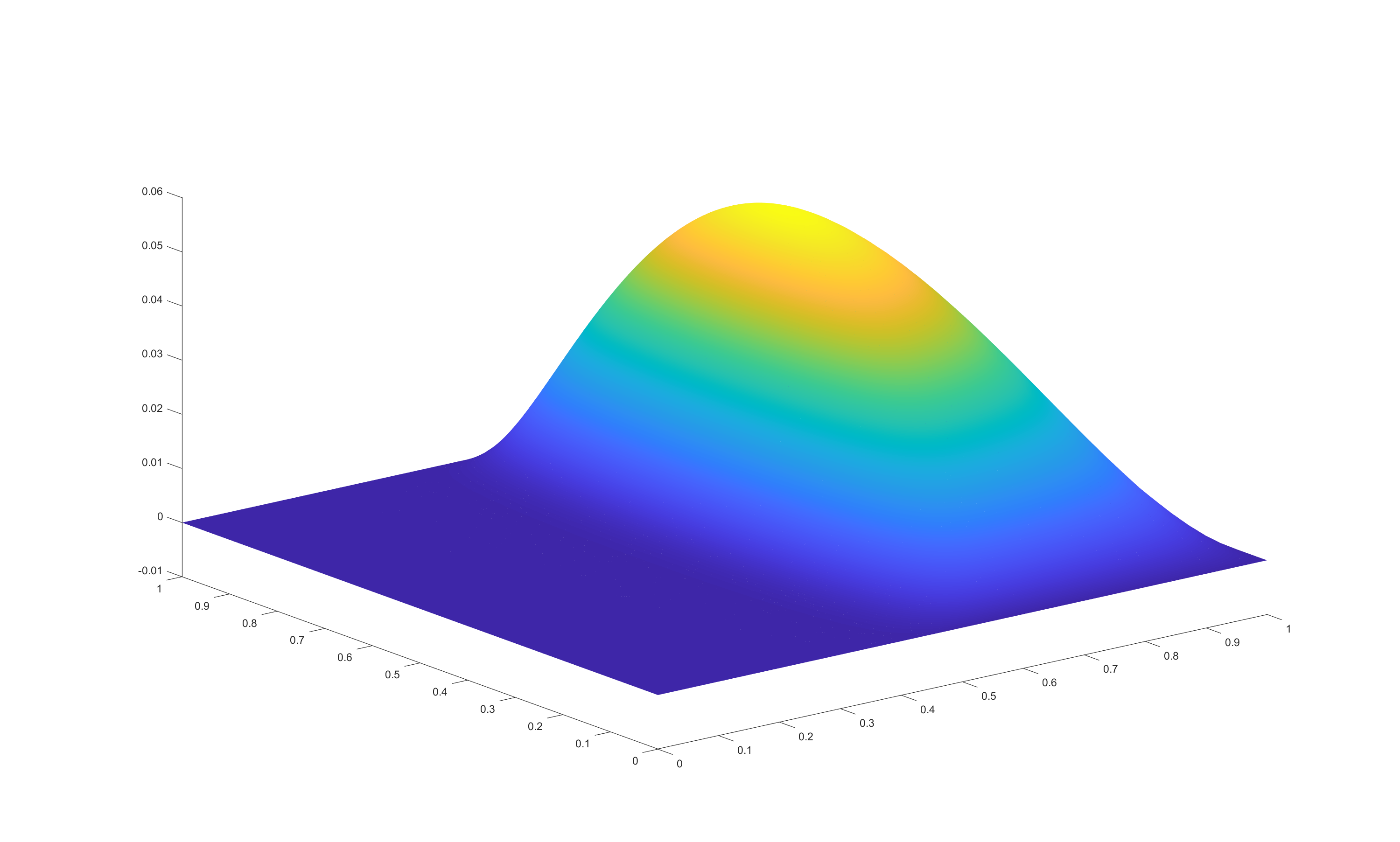} 
		\end{minipage}
		\label{Iu0.5}
	}
    	\subfigure[]{
    		\begin{minipage}[h]{0.48\textwidth}
		 	\includegraphics[width=1\textwidth]{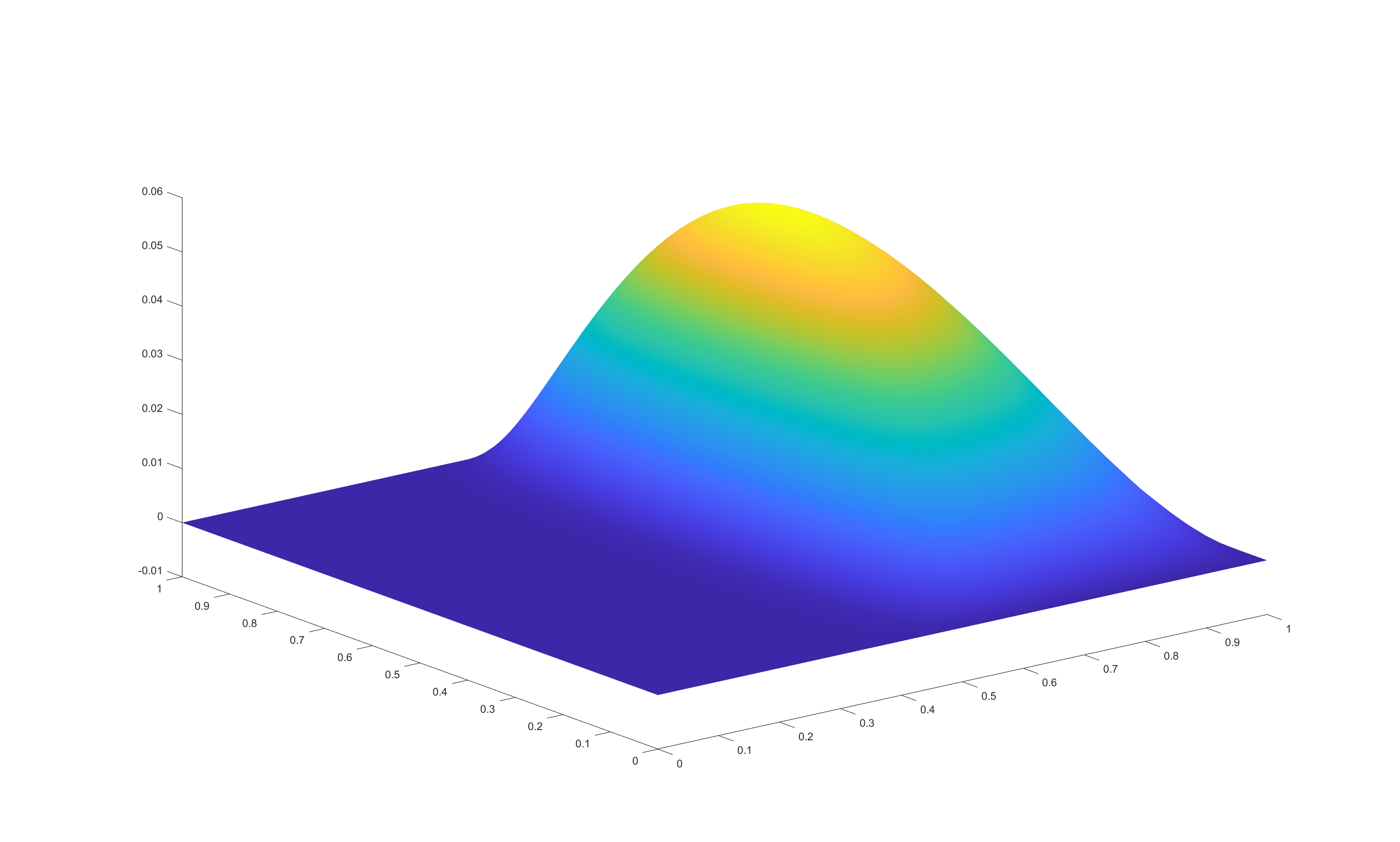}
    		\end{minipage}
		\label{Iuh0.5}
    	}
	\caption{(a) Convergence rates of the error and the error estimator; (b) The final adapted mesh; (c) Exact solution; (d) Numerical solution.}
	\label{fig2}
\end{figure}

\begin{example}\label{exam6.3}
  Let the exact solution be given by $u(x, y) = g(x)p(y)$, where the source term $f(x,y)$ is chosen accordingly and the component functions are defined as follows 
  \begin{align*}
    g(x)&=\frac{1}{2}\left[\sin(\pi x)+\frac{\pi \varepsilon}{1-e^{-1/ \varepsilon}}\left(e^{-x/ \varepsilon}+e^{(x-1)/ \varepsilon}-1-e^{-1/ \varepsilon}\right)\right],\\
    p(y)&=2y(1-y^{2})+\varepsilon\left[ld(1-2y)-3\frac{q}{l}+\left(\frac{3}{l}-d\right)e^{-y/\varepsilon}+\left(\frac{3}{l}+d\right)e^{(y-1)/ \varepsilon}\right].
  \end{align*}
  with the parameters $l = 1-e^{-1/ \varepsilon}$, $q=2-l$ and $d=1/(q-2\varepsilon l)$. 
\end{example}
In this test, we set $\varepsilon = 10^{-6}$ and $\theta = 0.5$. Figure \ref{fig3} shows the convergence history under adaptive refinement, the final adapted mesh, and comparisons between the exact and numerical solutions. 
\begin{figure}[h]
	\centering
	\subfigure[]{
		\begin{minipage}[h]{0.47\textwidth}
			\includegraphics[width=1\textwidth]{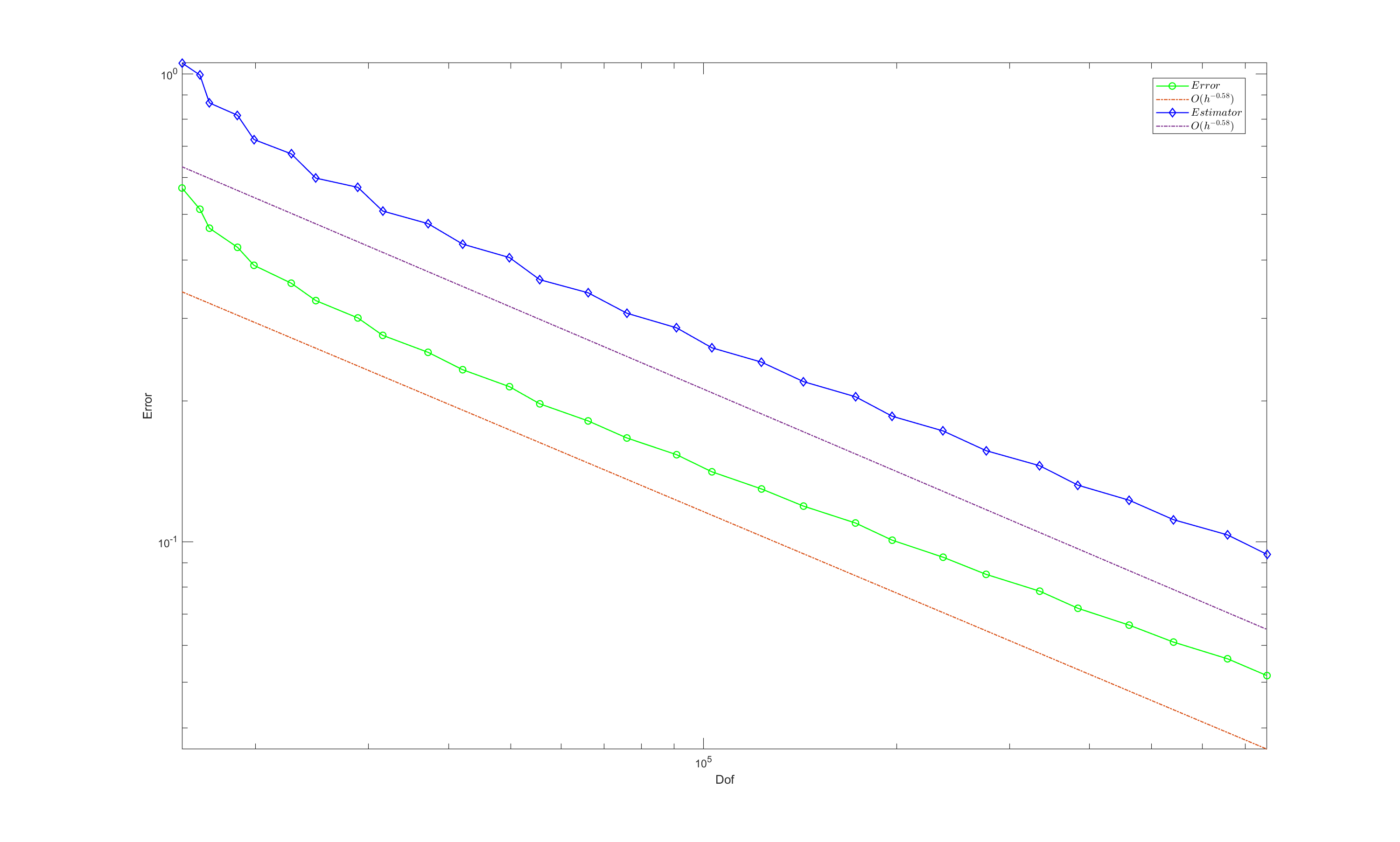} 
		\end{minipage}
		\label{S2c0.5}
	}
    	\subfigure[]{
    		\begin{minipage}[h]{0.47\textwidth}
   		 	\includegraphics[width=1\textwidth]{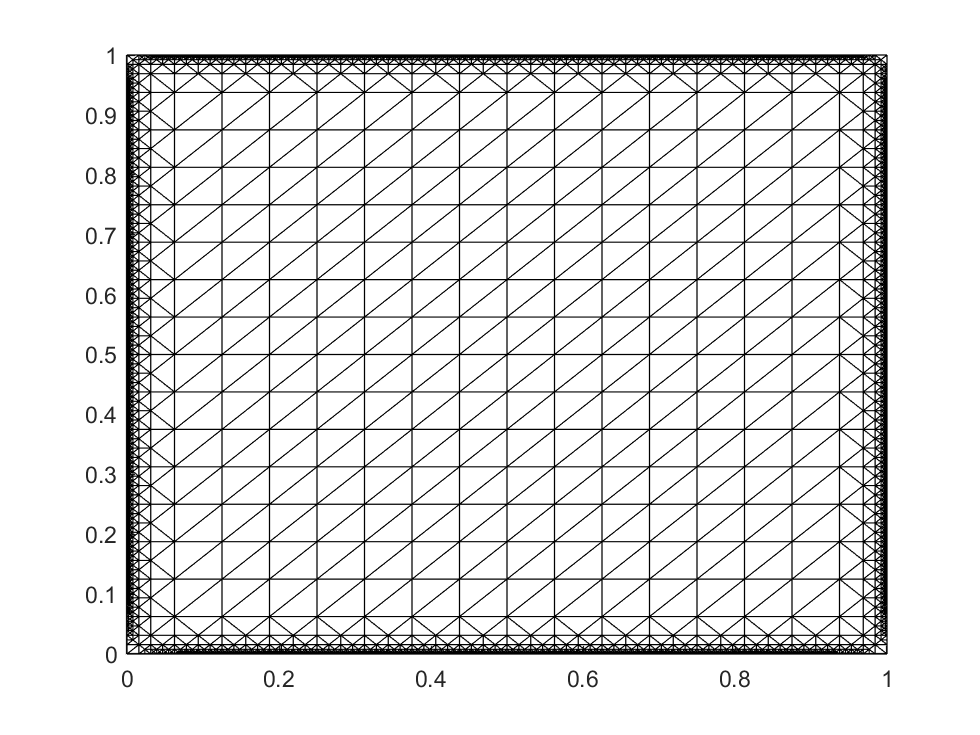}
    		\end{minipage}
		\label{S2m0.5}
    	}
	\\ 
	\subfigure[]{
		\begin{minipage}[h]{0.48\textwidth}
			\includegraphics[width=1\textwidth]{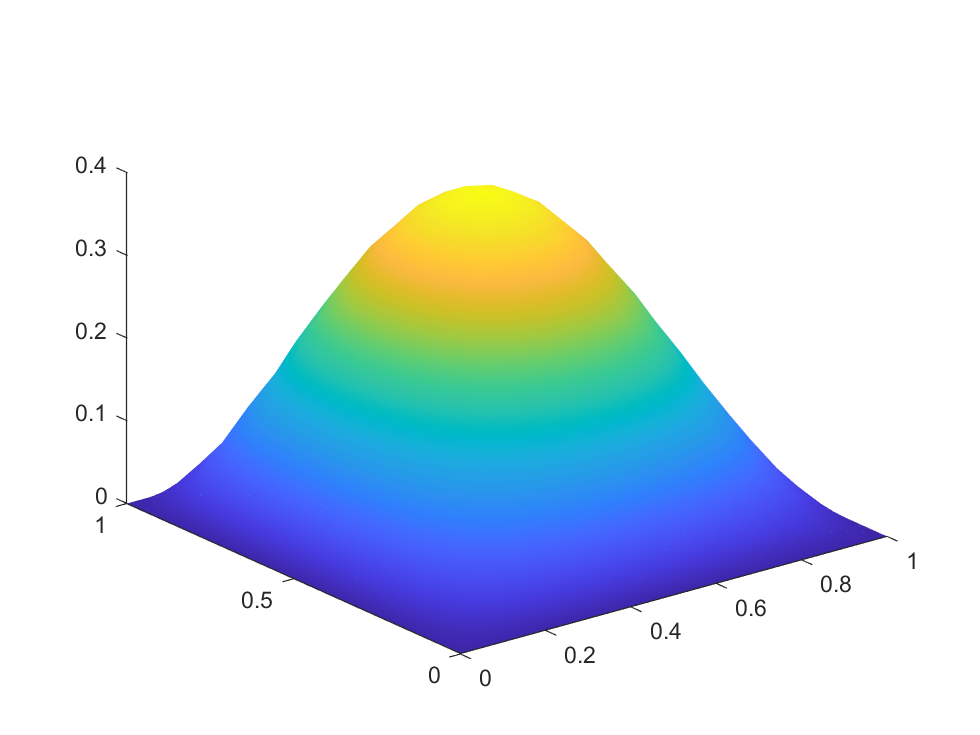} 
		\end{minipage}
		\label{S2u0.5}
	}
    	\subfigure[]{
    		\begin{minipage}[h]{0.48\textwidth}
		 	\includegraphics[width=1\textwidth]{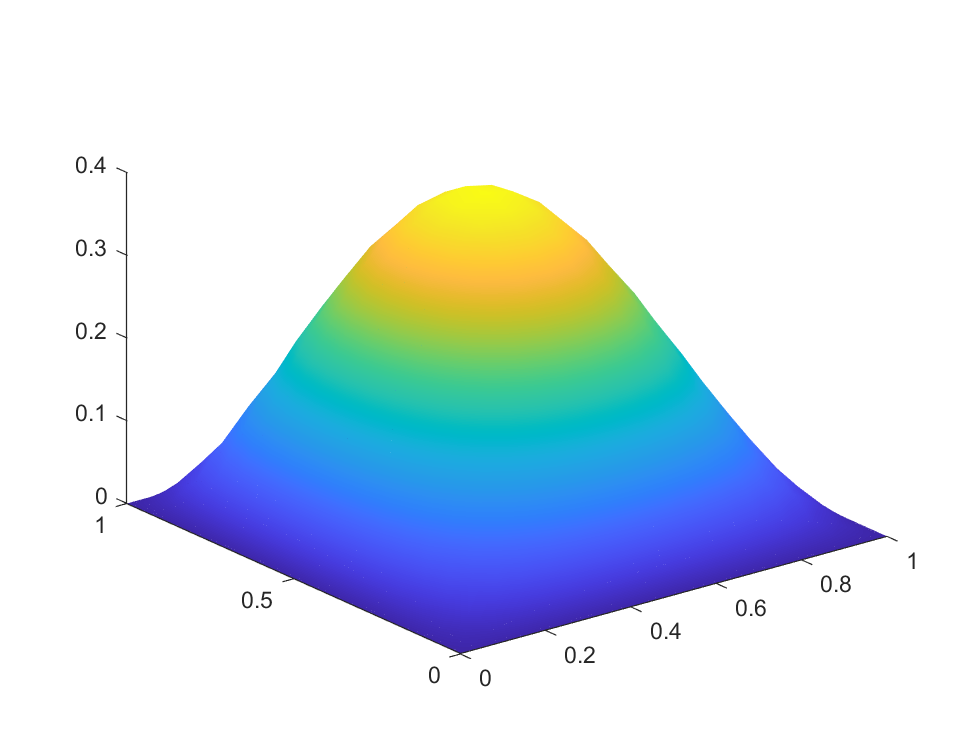}
    		\end{minipage}
		\label{S2uh0.5}
    	}
	\caption{(a) Convergence rates of the error and the error estimator; (b) The final adapted mesh; (c) Exact solution; (d) Numerical solution.}
	\label{fig3}
\end{figure}

\begin{example}
  This example follows the setup described in Han and Huang \cite{MR3039784}. Consider problem (\ref{1.1})-(\ref{1.2}) on the unit square $\Omega = (0, 1)^2$ with the source term
   \begin{align*}
    f(x, y) = 2\pi^{2}(1-\cos 2 \pi x \cos 2 \pi y).
   \end{align*}
\end{example}
Although the exact solution $u$ is not explicitly known, it is known to exhibit four sharp boundary layers near the edges of the domain. In the adaptive refinement procedure, the marking parameter is set to $\varepsilon = 10^{-6}$ and $\theta = 0.3$. 
\begin{figure}[h]
	\centering
	\subfigure[]{
		\begin{minipage}[h]{0.32\textwidth}
			\includegraphics[width=1\textwidth]{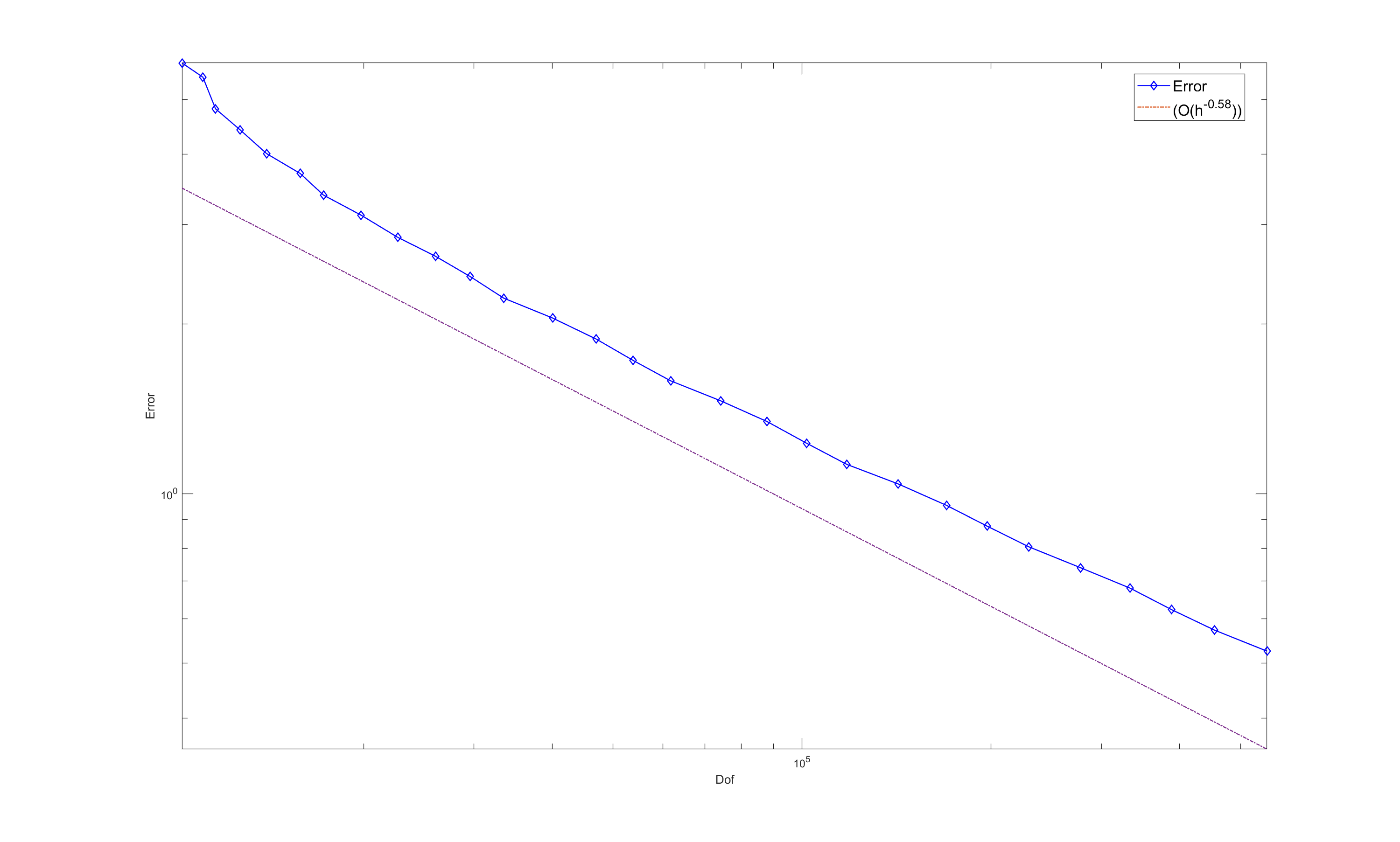} 
		\end{minipage}
		\label{Snc0.5}
	}
    	\subfigure[]{
    		\begin{minipage}[h]{0.30\textwidth}
   		 	\includegraphics[width=1\textwidth]{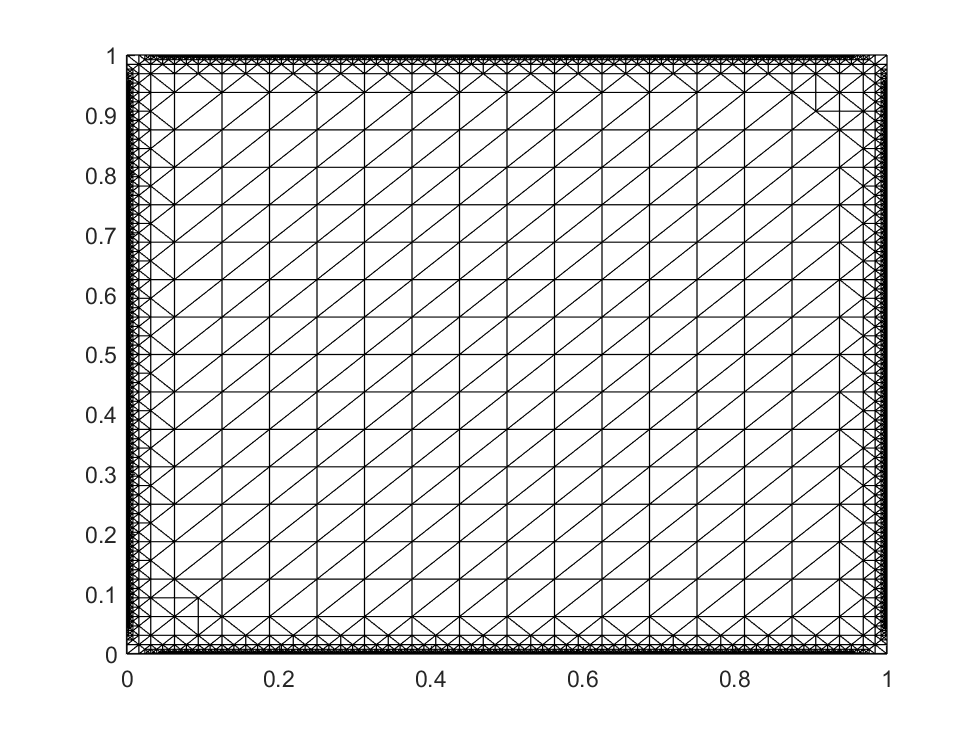}
    		\end{minipage}
		\label{Snm0.5}
    	}
    	\subfigure[]{
    		\begin{minipage}[h]{0.32\textwidth}
		 	\includegraphics[width=1\textwidth]{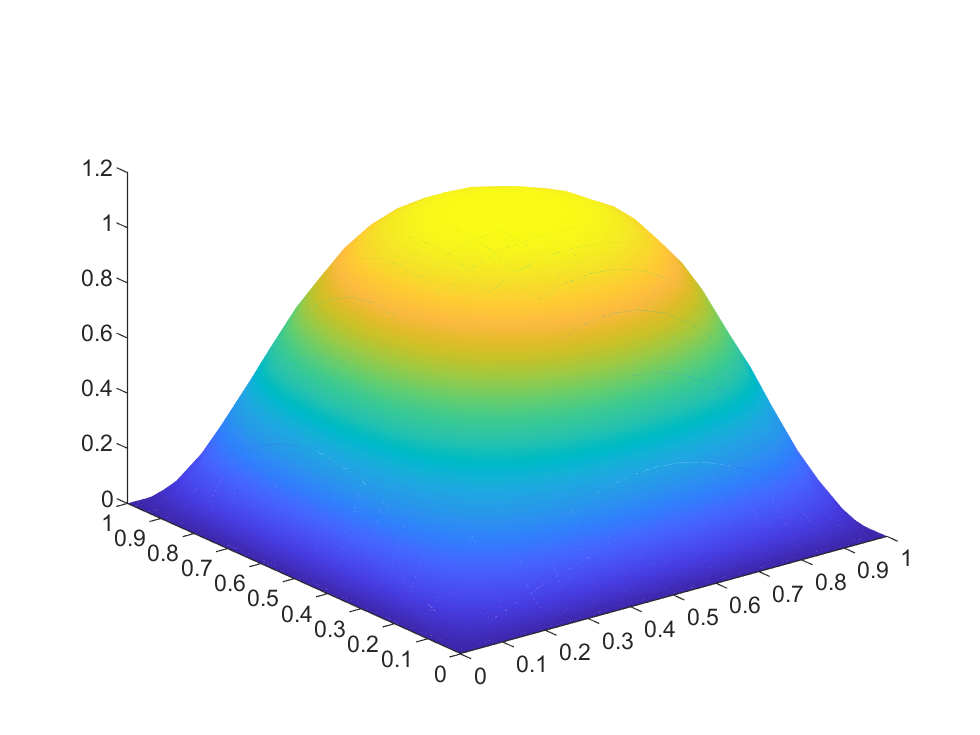}
    		\end{minipage}
		\label{Snuh0.5}
    	}
	\caption{(a) Convergence rates of the error and the error estimator; (b) The final adapted mesh; (c) Numerical solution.}
	\label{fig4}
\end{figure}


\smallskip
\noindent


\bibliographystyle{siam}
\bibliography{reference}

\newpage

\end{document}